\documentclass[11pt]{article}
\usepackage[T1]{fontenc}
\usepackage{lmodern,amsmath,amssymb,amsthm,mathtools}
\usepackage[margin=1in]{geometry}
\usepackage{microtype,booktabs,array,graphicx,pdflscape,longtable}
\usepackage{comment}
\usepackage{tikz}
\usetikzlibrary{arrows.meta}
\usepackage[colorlinks=true,linkcolor=blue,citecolor=blue,urlcolor=blue]{hyperref}
\numberwithin{equation}{section}
\newtheorem{theorem}{Theorem}[section]
\newtheorem{proposition}[theorem]{Proposition}
\newtheorem{lemma}[theorem]{Lemma}
\newtheorem{corollary}[theorem]{Corollary}
\newtheorem{maintheorem}[theorem]{Theorem}
\theoremstyle{definition}

\newtheorem{remark}[theorem]{Remark}
\newcommand{\N}{\mathbb N}
\newcommand{\R}{\mathbb R}
\newcommand{\Z}{\mathbb Z}
\newcommand{\T}{\mathbb T}

\newcommand{\Q}{\mathbb Q}
\newcommand{\Prob}{\mathcal P}
\newcommand{\ME}{\mathrel{\equiv_m}}
\newcommand{\meq}{\equiv_m}
\newcommand{\leB}{\mathrel{\leq_B}}
\newcommand{\nleB}{\mathrel{\nleq_B}}
\newcommand{\inc}{\mathrel{\perp_B}}
\newcommand{\Hel}{\operatorname{Hel}}
\newcommand{\Aff}{\operatorname{Aff}}
\newcommand{\E}{\mathbb E}
\newcommand{\id}{\operatorname{id}}

\newcommand{\QI}{\operatorname{QI}}
\newcommand{\UE}{\mathrel{\simeq_u}}
\newcommand{\Rep}{\operatorname{Rep}}

\title{Comparisons of benchmark equivalence relations}
\author{Ruijun Wang}
\date{}
\begin{document}
\maketitle
\begin{abstract}
We determine all Borel-reducibility comparisons within the family of
benchmark equivalence relations in Gao's monograph, enlarged by the
universal treeable relation $E_{\infty T}$ and its countable power
$E_{\infty T}^\omega$. We prove
$E_\infty^\omega\leq_B c_0$, $E_\infty\nleq_B\ell_p$ for
$2<p<\infty$, $\ell_p\nleq_B{\equiv_m}$ for $2<p<\infty$,
$c_0\nleq_B{\equiv_m}$, $c_0\nleq_B\ell_\infty^+$, and
$\equiv_m\nleq_B\ell_\infty^+$. Together with the classical result, these give $\ell_p\leq_B{\equiv_m}$ exactly when $p\leq2$.
The same methods show that
$E_{\infty T}$ is incomparable with every $\ell_p$, $1\leq p<\infty$. The comparisons answer questions recorded in Gao's monograph and Kanovei's monograph.
We also prove that unitary equivalence of all bounded linear operators
is not Borel reducible to $\equiv_m$, answering a question of Ding and Gao.
Equivalently, unitary equivalence of arbitrary representations of
$F_\infty$ is strictly more complicated than that of arbitrary
representations of $\mathbb Z$, answering a question of Thomas.
\end{abstract}

\section{Introduction}
\subsection{Classification and its history}
Many mathematical classification problems ask when two objects are the
same up to a specified equivalence: isomorphism of countable groups,
unitary equivalence of operators, or equality of measure classes.
A useful classification should assign complete invariants in a
definable way. Invariant descriptive set theory makes this requirement
precise by putting standard Borel structures on the spaces of objects
and comparing their equivalence relations by Borel reductions.

Part of the subject's history comes from representation theory.
Mackey's work on measurable spaces of representations made the quality
of a quotient Borel structure a central issue. Glimm's analysis of
type~I $C^*$-algebras and Effros's work on transformation groups linked
regularity of orbit spaces to the possibility of classification.
Effros's historical account \cite{Effros} describes this passage from
representation theory to general Polish group actions.
Silver's dichotomy \cite{Silver} and the
Harrington--Kechris--Louveau dichotomy \cite{HKL} gave basic dividing
lines for definable equivalence relations. Friedman and Stanley
\cite{FS} systematically developed Borel reducibility for classes of
countable structures. The theory of hyperfinite relations
\cite{DJK}, the general study of Polish group actions
\cite{BK}, and Hjorth's turbulence theory \cite{HjorthBook} supplied
further tools for distinguishing classification problems.

Gao's monograph \cite{Gao} brings these developments together.
Its Section~15.2 and Figure~15.1 organize several particularly useful
relations of low Borel complexity. We complete the comparison of that
specified benchmark family, include the universal treeable relation
$E_{\infty T}$ and its countable power, and study the position of
arbitrary bounded operators above measure
equivalence. The matrix is a comparison
of this family, not a classification of all $\Pi^0_3$ equivalence relations. We follow
Gao's symbols throughout: $\ell_p$ and $c_0$ denote the corresponding
relations on $\R^\N$, and $\meq$ denotes mutual absolute continuity
of probability measures.

\subsection{Statements and the questions they answer}
We recall the definitions in Section~\ref{sec:prelim}; here $\ell_p$
and $c_0$ denote coset equivalence on $\R^\N$, and
$\mu\equiv_m\nu$ means that the probability measures have the same
Borel null sets.

\begin{maintheorem}[Countable relations and sequence spaces]\label{main:countable}
\[
 E_\infty^\omega\leq_B c_0,
 \qquad E_\infty\inc\ell_p\quad(1\leq p<\infty).
\]
\end{maintheorem}
Kanovei asked whether $E_\infty$ reduces to $c_0$, to $\ell_1$,
or to some other $\ell_p$ \cite[Question~4.7]{Kanovei}.
Gao also records the $\ell_1$ question and asks whether $E_\infty$
reduces to an orbit relation of an abelian Polish group
\cite[Questions~2.3 and~2.5]{GaoBanach}.
The broad question about some abelian Polish-group action was already
answered positively in the announced work of Frisch and Shinko
\cite[slide~8]{FrischShinkoSlides}; an earlier announcement was made
in October~2023 \cite{FrischShinkoAnnouncement}.
Our reduction supplies the specific natural coset target $c_0$.
The distinction between a specified coset relation and an arbitrary
action of the same group is essential here.
Thomas \cite[Theorem~6.12]{Thomas} already proved
$E_\infty\nleq_B\ell_2$, crediting an observation of Kechris using
Popa's superrigidity theorem; Dougherty--Hjorth \cite{DH} then give
the same conclusion for $1\leq p\leq2$.
Our spectral-gap argument covers every finite exponent and supplies
the remaining range $p>2$. The reverse nonreductions are classical
consequences of Hjorth's turbulence theory.
The passage from $E_\infty$ to $E_\infty^\omega$ uses the known
identity $(c_0)^\omega\sim_B c_0$ \cite[Lemma~7.3]{GaoBanach}.

\begin{maintheorem}[The universal treeable relation]\label{main:treeable}
For every $1\leq p<\infty$,
\[
 E_{\infty T}<_B c_0,\qquad E_{\infty T}<_B{\equiv_m},\qquad
 E_{\infty T}\inc\ell_p.
\]
Moreover, $E_{\infty T}$ is incomparable with each of
$E_0^\omega$, $E_1$, and $E_1^\omega$.
\end{maintheorem}
The classical position $E_0<_B E_{\infty T}<_B E_\infty$
comes from Jackson--Kechris--Louveau \cite[Section~3.5]{JKL}.
The finite-exponent assertion is a direct application of our
spectral-gap theorem to the free Bernoulli shift; nonreducibility
of $E_\infty$ alone would not imply it.
The tail comparison uses Miller's theorem \cite[Theorem~4]{Miller}.
We give the deductions in Section~\ref{sec:treeable} and distinguish
these coset relations from the arbitrary $\ell_1$-actions in Allison's
theorem \cite[Theorem~1.2]{Allison}.

\begin{maintheorem}[The countable power of the universal treeable relation]
\label{main:treeable-power}
\[
 E_0^\omega<_B E_{\infty T}^\omega<_B E_\infty^\omega,
 \qquad E_{\infty T}^\omega\inc E_\infty.
\]
The relation $E_{\infty T}^\omega$ is not Borel bireducible with any
other member of the benchmark family displayed below.
\end{maintheorem}
These classical separations use the property-$(\mathrm T)$ obstruction
to countable products of treeable relations, explicitly recorded by
Jackson--Kechris--Louveau \cite[Section~3.7, after Theorem~3.31]{JKL}.
They credit the methods of Adams--Spatzier and the formulation of
Hjorth--Kechris \cite{AdamsSpatzier,HK1996}.
Section~\ref{sec:treeable-power} gives the resulting complete comparison
profile. We include this additional benchmark with its established
provenance.

\begin{maintheorem}[The exponent threshold]\label{main:lp}
For $1\leq p\leq\infty$,
\[
 \ell_p\leq_B{\equiv_m}\quad\Longleftrightarrow\quad p\leq2.
\]
\end{maintheorem}
\begin{maintheorem}[Null sequences and measure classes]\label{main:c0}
\[c_0\nleq_B{\equiv_m}.\]
\end{maintheorem}
These answer the questions recorded in Gao
\cite[Section~8.5, p.~200; Section~15.2, p.~349]{Gao}.
The positive direction in Theorem~\ref{main:lp} is classical:
combine Dougherty--Hjorth with Kakutani's product-measure criterion
\cite{DH,Kakutani}. The endpoint $p=\infty$ is also classical,
by $E_1\leq_B\ell_\infty$ and the orbit-relation obstruction
of Kechris--Louveau \cite{KL}. The contribution here is the negative
direction for finite $p>2$ and Theorem~\ref{main:c0}.
We give the compact-averaging argument and
the geometric obstruction in full. The Polish topology on a
quasi-invariance group and its connection with unitary representations
are established ingredients \cite{HMP,Gabri06}.

\begin{maintheorem}[Two jump obstructions]\label{main:jumps}
\[
 c_0\nleq_B\ell_\infty^+,
 \qquad {\equiv_m}\nleq_B\ell_\infty^+.
\]
\end{maintheorem}
These settle two comparisons left undecided by Gao's original diagram.
Theorems~\ref{main:countable}--\ref{main:jumps},
together with the cited comparisons collected in
Section~\ref{sec:established}, determine every entry of
Table~\ref{tab:matrix}.

Fix a separable infinite-dimensional complex Hilbert space
$\mathcal H$. Let $\simeq_u$ denote unitary equivalence of all bounded
operators on $\mathcal H$. In Thomas's notation,
$\approx_\Gamma^+$ denotes unitary equivalence of arbitrary
representations of the countable group $\Gamma$ on $\mathcal H$.
Here the superscript denotes arbitrary representations; it is not
the jump used in Theorem~\ref{main:jumps}.
\begin{maintheorem}[Bounded operators]\label{main:operators}
\[{\equiv_m}<_B{\simeq_u}.\]
\end{maintheorem}
\begin{maintheorem}[Arbitrary representations]\label{main:thomas}
\[
 {\approx^+_{\Z}}\sim_B{\equiv_m}
 <_B{\approx^+_{F_\infty}}\sim_B{\simeq_u}.
\]
\end{maintheorem}

The superscript here is
Thomas's notation for arbitrary representations, \emph{not} a jump operation.

Ding and Gao ask whether unitary equivalence of all bounded operators
is Borel reducible to measure equivalence \cite[p.~735]{DingGao}.
Thomas asks the corresponding comparison for arbitrary representations
of $F_\infty$ and $\Z$ \cite[Question~8.2, p.~1784]{Thomas}.
Theorems~\ref{main:operators} and~\ref{main:thomas} give negative
answers to these equivalent formulations. The lower reduction from
measure equivalence and the comparison for $\Z$ are classical.

The meaning of the operator conclusion requires care. Hjorth and
T\"ornquist proved that unitary equivalence of arbitrary representations
is $\Pi^0_3$ \cite{HT12}; this pointclass bound does not supply a
reduction to measure equivalence. Recent work of Lupini
\cite[Theorems~F and~G]{Lupini26} concerns irreducible representations
and a different comparison of unitary duals. Our conclusion rules
out Borel complete invariants given by scalar measure classes for
all bounded operators. It makes no assertion against other forms of
spectral analysis, functional calculus, or operator-valued invariants.
This is consistent with Ding and Gao's distinction between spectral
theory and classification \cite[p.~731]{DingGao}.

We give full proofs of the additional comparisons above and cite the
established inputs at each use. Together they complete the comparison
in the following family.

\subsection{Scope and the comparison matrix}
The comparison theorem concerns Gao's benchmark family together with
the universal treeable relation and its countable power:
\[
\begin{gathered}
\id(\N),\quad \id(2^\N),\quad E_0,\quad E_0^\omega,\quad
E_{\infty T},\quad E_{\infty T}^\omega,\\
E_\infty,\quad E_\infty^\omega,\quad =^+,\quad c_0,\quad
\ell_r\ (1\le r<\infty),\quad \meq,\\
E_1,\quad\ell_\infty,\quad E_1^\omega,\quad E_1^+,\quad\ell_\infty^+ .
\end{gathered}
\]
Here $\id(2^\N)$ is Borel isomorphic to $\id(\R)$.
The definitions are given in Section~2. Some members already have
simpler representatives, for example $F_\sigma$ representatives; the
designation $\Pi^0_3$ is an upper bound for the benchmark presentations,
not a claim of exact Borel rank for every vertex.

Table~\ref{tab:matrix} gives every ordered comparison for a fixed
$2<p<\infty$. A row is the source and a column is the target:
$Y$ means a Borel reduction exists, $N$ means none exists, and the
diagonal symbol $=$ denotes the identity reduction. The table is
independent of the chosen value of $p$ in this range.
It describes the full parameterized family by the following rules:
\begin{enumerate}
\item For finite $r,s$, $\ell_r\le_B\ell_s$ exactly when $r\le s$,
by Dougherty--Hjorth \cite{DH}.
\item Against every vertex outside the finite-exponent family, all
$\ell_r$ with $1\le r\le2$ have the same comparison profile as $\ell_2$.
\item Against those same vertices, all $\ell_r$ with $2<r<\infty$
have the profile of the displayed $\ell_p$.
\end{enumerate}
Thus the matrix together with these rules includes parameters
$1<r<2$ and comparisons between two distinct parameters above~$2$.
In particular,
\begin{equation}\label{eq:intro-infty-lp}
 E_{\infty T}\inc\ell_r,\qquad E_{\infty T}^\omega\inc\ell_r,\qquad E_\infty\inc\ell_r\quad(1\le r<\infty),
\end{equation}
where $\inc$ means that neither relation reduces to the other.

\begin{table}[htbp]
\centering
\small
\setlength{\tabcolsep}{3pt}
\renewcommand{\arraystretch}{1.13}
\resizebox{\textwidth}{!}{%
\begin{tabular}{r@{\quad}l*{19}{c}}
\toprule
\multicolumn{2}{c}{Source $\backslash$ target} & 01 & 02 & 03 & 04 & 05 & 06 & 07 & 08 & 09 & 10 & 11 & 12 & 13 & 14 & 15 & 16 & 17 & 18 & 19\\
\midrule
01 & $\operatorname{id}(\omega)$ & = & Y & Y & Y & Y & Y & Y & Y & Y & Y & Y & Y & Y & Y & Y & Y & Y & Y & Y\\
02 & $\operatorname{id}(2^\omega)$ & N & = & Y & Y & Y & Y & Y & Y & Y & Y & Y & Y & Y & Y & Y & Y & Y & Y & Y\\
03 & $E_0$ & N & N & = & Y & Y & Y & Y & Y & Y & Y & Y & Y & Y & Y & Y & Y & Y & Y & Y\\
04 & $E_0^\omega$ & N & N & N & = & N & Y & N & Y & Y & Y & N & N & N & Y & N & N & Y & Y & Y\\
05 & $E_{\infty T}$ & N & N & N & N & = & Y & Y & Y & Y & Y & N & N & N & Y & N & Y & N & Y & Y\\
06 & $E_{\infty T}^\omega$ & N & N & N & N & N & = & N & Y & Y & Y & N & N & N & Y & N & N & N & Y & Y\\
07 & $E_\infty$ & N & N & N & N & N & N & = & Y & Y & Y & N & N & N & Y & N & Y & N & Y & Y\\
08 & $E_\infty^\omega$ & N & N & N & N & N & N & N & = & Y & Y & N & N & N & Y & N & N & N & Y & Y\\
09 & $=^+$ & N & N & N & N & N & N & N & N & = & N & N & N & N & Y & N & N & N & Y & Y\\
10 & $c_0$ & N & N & N & N & N & N & N & N & N & = & N & N & N & N & N & N & N & N & N\\
11 & $\ell_1$ & N & N & N & N & N & N & N & N & N & N & = & Y & Y & Y & N & Y & N & N & Y\\
12 & $\ell_2$ & N & N & N & N & N & N & N & N & N & N & N & = & Y & Y & N & Y & N & N & Y\\
13 & $\ell_p$ & N & N & N & N & N & N & N & N & N & N & N & N & = & N & N & Y & N & N & Y\\
14 & $\equiv_m$ & N & N & N & N & N & N & N & N & N & N & N & N & N & = & N & N & N & N & N\\
15 & $E_1$ & N & N & N & N & N & N & N & N & N & N & N & N & N & N & = & Y & Y & Y & Y\\
16 & $\ell_\infty$ & N & N & N & N & N & N & N & N & N & N & N & N & N & N & N & = & N & N & Y\\
17 & $E_1^\omega$ & N & N & N & N & N & N & N & N & N & N & N & N & N & N & N & N & = & Y & Y\\
18 & $E_1^+$ & N & N & N & N & N & N & N & N & N & N & N & N & N & N & N & N & N & = & Y\\
19 & $\ell_\infty^+$ & N & N & N & N & N & N & N & N & N & N & N & N & N & N & N & N & N & N & =\\
\bottomrule
\end{tabular}}
\caption{All ordered comparisons in the nineteen-vertex slice, for any fixed $2<p<\infty$. Each column has the same numerical key as the corresponding row.}
\label{tab:matrix}
\end{table}

The rest of this paper is organized as follows. Section~\ref{sec:prelim} supplies the descriptive, analytic, and
operator-algebraic preliminaries. Section~\ref{sec:established}
derives the matrix from established comparisons and the subsequent
theorems. Section~\ref{sec:sequence-measure} proves the measure
threshold and the null-sequence obstruction;
Section~\ref{sec:countable-sequences} treats countable relations and
the treeable benchmark and its countable power;
Section~\ref{sec:jumps} proves the jump obstructions;
Section~\ref{sec:operators} proves the operator and representation
theorems.
\section{Preliminaries}\label{sec:prelim}

\subsection{Borel spaces, reductions, and orbit relations}
A \emph{Polish space} is a separable completely metrizable topological
space. A \emph{standard Borel space} is a set with the Borel
$\sigma$-algebra of a Polish topology. A map is \emph{Borel} if the
inverse image of every Borel set is Borel.
A Borel equivalence relation on $X$ is an equivalence relation
that is a Borel subset of $X\times X$.
For relations $E$ on $X$ and $F$ on $Y$, a \emph{Borel reduction}
is a Borel map $f:X\to Y$ such that
\[
 x\,E\,x'\quad\Longleftrightarrow\quad f(x)\,F\,f(x').
\]
We write $E\le_B F$ if such a map exists,
$E\sim_B F$ if both reductions exist, and $E<_B F$ if
$E\le_B F$ but $F\nleq_B E$. A \emph{homomorphism} only
requires the forward implication. Reducibility is transitive;
incomparability is denoted by $E\inc F$.

A relation is \emph{smooth} if it reduces to equality on a standard
Borel space. A relation is \emph{countable} if every class is countable.
A countable Borel relation is \emph{hyperfinite} if it is an increasing
union of Borel equivalence relations with finite classes.
A Borel relation is \emph{essentially countable} if it Borel reduces
to a countable Borel equivalence relation. A countable Borel relation
is \emph{treeable} if it has an acyclic Borel graph whose connected
components are precisely its classes. Such a graph is called a
\emph{treeing} \cite{JKL}.
A Borel relation is \emph{hypersmooth} if it is an increasing
union of smooth Borel equivalence relations.
These are different notions. A relation is universal for a class if every member of that class Borel reduces to it.

An action $G\curvearrowright X$ of a Polish group induces the
\emph{orbit equivalence relation}
$x\,E_G^X\,y$ if $g\cdot x=y$ for some $g\in G$.
Relations reducible to isomorphism of countable structures are
said to be \emph{classifiable by countable structures}.
Equivalently, they reduce to an orbit relation of the permutation
group $S_\infty$; see \cite{BK,Gao}.

The Borel hierarchy begins with open sets ($\Sigma^0_1$) and
closed sets ($\Pi^0_1$). A $\Sigma^0_2$, or $F_\sigma$, set is
a countable union of closed sets. A $\Pi^0_3$ set is a countable
intersection of $F_\sigma$ sets. A $K_\sigma$ set is a
countable union of compact sets. The latter is a more restrictive
condition on a given noncompact ambient space.
A $G_\delta$ set is a countable intersection of open sets; a
$G_\delta$ subspace of a Polish space is Polish.

\subsection{The benchmark relations}
We use $\N=\{0,1,\ldots\}$. Equality on $X$ is $\id(X)$.
The relation $E_0$ on $2^\N$ is eventual equality:
$x\,E_0\,y$ when $x(n)=y(n)$ for all sufficiently large $n$.
The relation $E_1$ is eventual equality of rows on
$(2^\N)^\N$: $x\,E_1\,y$ when $x_n=y_n$ for all sufficiently
large $n$. It is the universal hypersmooth Borel relation
\cite{KL,Gao}.

The relation $E_\infty$ is universal among countable Borel equivalence
relations. We use the full shift action of the free group
$\mathbb F_2$ on $2^{\mathbb F_2}$:
$(g\cdot x)(h)=x(g^{-1}h)$. The orbit relation of this full
shift is a model of $E_\infty$ \cite{Gao}.
The restriction to the invariant free part
\[
 X_T=\operatorname{Free}(2^{\mathbb F_2})
     =\{x:(\forall g\ne1)\ g\cdot x\ne x\}
\]
is the universal treeable relation $E_{\infty T}$
\cite[Theorem~3.17 and Definition~3.18]{JKL}; see also
\cite[Definition~2.6]{HK2001} for this concrete model.
The set $X_T$ is $G_\delta$, and the generator graph restricted to it
is a treeing of degree four. Thus $X_T$ is Polish and $E_{\infty T}$
has an $F_\sigma$ presentation, as a countable union of closed action
graphs on $X_T$. Its universality is among countable treeable Borel
relations, whereas $E_\infty$ is universal among all countable Borel
relations.

For a relation $E$ on $X$, its countable product and its
Friedman--Stanley jump are respectively
\[
 (x_n)\,E^\omega\,(y_n)\ \Longleftrightarrow\
       (\forall n)\ x_n\,E\,y_n,
 \qquad
 (x_n)\,E^+\,(y_n)\ \Longleftrightarrow\
       \{[x_n]_E:n\in\N\}=\{[y_n]_E:n\in\N\}.
\]
Thus the jump discards order and multiplicity.
The relation $=^+$ is equality of nonempty countable sets of reals,
coded by sequences; using Cantor space instead of $\R$ gives a
bireducible relation. The symbols $E_0^\omega$, $E_{\infty T}^\omega$,
$E_\infty^\omega$, $E_1^\omega$, $E_1^+$, and $\ell_\infty^+$ have the
above meanings. In particular, $E_{\infty T}^\omega$ is the relation on
$X_T^\N$ given by
\[
 (x_n)\,E_{\infty T}^\omega\,(y_n)
 \quad\Longleftrightarrow\quad
 (\forall n\in\N)\ x_n\,E_{\infty T}\,y_n.
\]
It is the orbit relation of the coordinatewise continuous action of
the Polish group $(\mathbb F_2)^\N$ on $X_T^\N$. Since each coordinate
condition is $F_\sigma$, this gives a $\Pi^0_3$ presentation.
The relation is universal for countable products of treeable countable
Borel relations: choose a reduction of each factor to $E_{\infty T}$
and take the product of these maps. Reindexing $\N\times\N$ by $\N$
also shows $(E_{\infty T}^\omega)^\omega\sim_B E_{\infty T}^\omega$.

Following Gao, we use the sequence-space symbols also for their
coset relations on \emph{all} of $\R^\N$:
\[
\begin{aligned}
x\,\ell_r\,y &\ \Longleftrightarrow\
       \sum_n|x_n-y_n|^r<\infty &&(1\le r<\infty),\\
x\,c_0\,y &\ \Longleftrightarrow\
       |x_n-y_n|\longrightarrow0,\\
x\,\ell_\infty\,y &\ \Longleftrightarrow\
       \sup_n|x_n-y_n|<\infty .
\end{aligned}
\]
The same symbols denote the usual Banach spaces when used for
vectors or norms. Merely including one Banach space in another
does not give a reduction between these coset relations:
a reduction must also preserve \emph{inequivalence}.
For $t\ge0$ let
\[
 L_t=\{(a,b)\in(\R^\N)^2:(\forall i)\ |a_i-b_i|\le t\}.
\]
These sets are closed,
$\ell_\infty=\bigcup_{n\in\N}L_n$, and
$L_t\circ L_u\subseteq L_{t+u}$.

Let $C=2^\N$. The space $\Prob(C)$ of Borel probability measures
has its Polish weak topology: $\mu_j\to\mu$ means convergence
of integrals of all continuous real functions.
Write $\mu\ll\nu$ when every $\nu$-null Borel set is $\mu$-null,
and $\mu\meq\nu$ when both $\mu\ll\nu$ and $\nu\ll\mu$.
Borel isomorphisms between uncountable standard Borel sample spaces
transport this equivalence relation.
Its $\Pi^0_3$ description follows from uniform absolute continuity:
\[
\mu\ll\nu\ \Longleftrightarrow\
\forall r\ge1\ \exists s\ge1\ \forall B\text{ clopen in }C:
\quad \nu(B)<2^{-s}\ \Longrightarrow\ \mu(B)\le2^{-r}.
\]
For fixed $r,s,B$ the condition is closed.
The clopen algebra suffices by regularity and approximation.
The corresponding definitions similarly exhibit $c_0$,
$E_0^\omega$, $E_1^\omega$, $E_1^+$, and $\ell_\infty^+$
as $\Pi^0_3$ in their standard presentations; see \cite{Gao}.

\subsection{Measure-theoretic notation}
For a probability $\mu$, $L^r(\mu)$ consists of measurable real
functions, identified when equal almost everywhere, with finite
$\int|f|^r\,d\mu$; its norm is
$(\int|f|^r\,d\mu)^{1/r}$. In particular $L^2(\mu)$ is a
Hilbert space with inner product $\int fg\,d\mu$.
If $\nu\ll\mu$, its Radon--Nikodym density $d\nu/d\mu$
is the nonnegative measurable function satisfying
$\nu(B)=\int_B(d\nu/d\mu)\,d\mu$.
Mutual absolute continuity means that this density is positive
almost everywhere. We write $f^2\mu$ for the measure with density
$f^2$, which is a probability when $\|f\|_2=1$.

A measure is \emph{atomless} if every singleton has measure zero,
and has \emph{full support} if every nonempty open set has positive
measure. For probabilities $\mu,\nu$, the Hellinger distance is
\[
 \Hel(\mu,\nu)=
 \left\|\sqrt{d\mu/d\lambda}-\sqrt{d\nu/d\lambda}\right\|_{L^2(\lambda)},
\]
where any probability dominating both may be used; the value is
independent of this choice. We use the total variation norm with
the convention
$\|\mu-\nu\|_{\mathrm{TV}}=\int|d\mu/d\lambda-d\nu/d\lambda|\,d\lambda$,
so it is at most $2$ for probabilities.
The elementary inequality
$\Hel(\mu,\nu)^2\le\|\mu-\nu\|_{\mathrm{TV}}$
will be used below.

\subsection{Graphs and category}
A \emph{Borel graph} is a symmetric irreflexive Borel subset of
$X\times X$. Its graph distance $d_G(x,y)$ is the shortest path
length, with value $\infty$ between different components.
Bounded finite degree makes bounded graph neighborhoods of Borel
sets Borel, by the Lusin--Novikov theorem.
We shall use without proof the finite coloring theorem
\cite[Proposition~4.6]{KST}: a Borel graph of degree at most
$d<\infty$ has a Borel proper coloring with $d+1$ colors.

A set is \emph{meager} if it is a countable union of nowhere dense
sets; its complement is \emph{comeager}.
A set has the \emph{Baire property} if it differs from an open set
by a meager set. We use standard category results from \cite{Gao}:
Borel sets have the Baire property; a Borel map between Polish
spaces is continuous on a comeager subspace; continuous open maps
pull back comeager sets to comeager sets; and the Kuratowski--Ulam
theorem gives comeager sections of a comeager subset of a product.
An equivalence relation is \emph{generically ergodic} if every
invariant set with the Baire property is meager or comeager.
A continuous action all of whose orbits are dense has this property.

\subsection{Spectral gap and scalar Poincare inequalities}
A probability-measure-preserving action is \emph{ergodic} if every
invariant measurable set has measure zero or one.
A probability-measure-preserving action
$\Gamma\curvearrowright(X,\mu)$ has \emph{spectral gap} if there are
a finite symmetric set $S\subseteq\Gamma$ and $c>0$ such that
\[
 \sum_{s\in S}\|u\circ s-u\|_2^2
       \ge c\|u\|_2^2
 \quad\text{whenever }u\in L^2(X,\mu),\ \int u\,d\mu=0.
\]
This inequality implies ergodicity, by applying it to
$u=1_A-\mu(A)$ for an invariant measurable set $A$.
A \emph{median} of a real random variable $v$ is a number $m$
with $\mu(v>m)\le1/2$ and $\mu(v<m)\le1/2$.
Every finite real random variable has a finite median.
The standard scalar $L^r$ form of spectral gap gives, for each
$1\le r<\infty$, a finite symmetric set $S_r$ and $C_r<\infty$ with
\begin{equation}\label{eq:scalar-poincare}
 \int |u|^r\,d\mu
 \le C_r\sum_{s\in S_r}\int |u(sx)-u(x)|^r\,d\mu(x)
\end{equation}
for every bounded real $u$ with median zero.
We use the equivalence of scalar spectral gaps in the different
$L^r$ spaces from Bader--Furman--Gelander--Monod
\cite[Remark~4.3]{BFGM}.
Their formulation is modulo invariant vectors, which for an
ergodic action are constants. For median-zero $u$,
$\|u\|_r\le(1+2^{1/r})\inf_{a\in\R}\|u-a\|_r$,
so it gives \eqref{eq:scalar-poincare}.
The constants are independent of the number of coordinates when
this scalar inequality is summed over a finite vector.
The Bernoulli action of $\mathbb F_2$ with fair product measure
has spectral gap; see \cite[Remark~6.4]{KidaTuckerDrob}.

\subsection{Property \texorpdfstring{$(\mathrm T)$}{(T)} and treeable targets}
\label{sec:property-T}
A countable group $\Gamma$ has \emph{Kazhdan's property
$(\mathrm T)$} if every unitary representation $\pi$ with almost
invariant unit vectors has a nonzero invariant vector. Here almost
invariance means that for every finite $Q\subseteq\Gamma$ and
$\varepsilon>0$ there is a unit vector $v$ satisfying
$\|\pi(g)v-v\|<\varepsilon$ for all $g\in Q$.
The groups $\mathrm{SL}_n(\Z)$, $n\ge3$, have this property;
see \cite[Section~3.7]{JKL}.

Let $R$ be a countable Borel equivalence relation on a probability
space $(X,\mu)$ and let $F$ be a Borel equivalence relation on $Y$.
We say that $(R,\mu)$ is \emph{$F$-ergodic} if every Borel
homomorphism $f:R\to F$ maps a $\mu$-conull Borel set into one
$F$-class. Ordinary ergodicity is the
case where $F$ is equality on $2^\N$. The following established
result is used without reproducing its rigidity proof.

\begin{theorem}[Adams--Spatzier; Hjorth--Kechris]
\label{thm:propertyT-treeable}
Let a countable group $\Gamma$ with property $(\mathrm T)$ act by
probability-preserving Borel transformations on $(X,\mu)$, where
$\mu$ is atomless and the action is ergodic. Its orbit relation $R$
is $F$-ergodic for every treeable countable Borel relation $F$.
Consequently $R\nleq_B F^\omega$ for every such $F$.
\end{theorem}
The first assertion is the formulation of the Adams--Spatzier
methods \cite{AdamsSpatzier} isolated in
\cite[Theorem~10.5]{HK1996}; see also
\cite[Theorem~3.7]{Calderoni}.
The countable-product conclusion is explicitly stated in
\cite[Section~3.7, immediately after Theorem~3.31]{JKL}.
It follows by intersecting the conull sets obtained for the countably
many coordinate homomorphisms: any proposed reduction would put a
conull set into one countable source class.

\subsection{Polishable subgroups}
A Polish group is a group with a Polish topology making its group
operations continuous. A subgroup of an ambient Polish group is
Polishable here if it admits a Polish group topology with continuous
inclusion. We use the Lusin--Souslin theorem and automatic continuity
of Borel homomorphisms between Polish groups
\cite[Theorems~1.3.1 and~2.3.3]{Gao}.
\begin{lemma}\label{lem:unique-topology}
Suppose a subgroup $H$ of a Polish group $K$ is equipped with two Polish
group topologies, and both inclusion maps into $K$ are continuous.
Then the two topologies on $H$ coincide.
\end{lemma}
\begin{proof}
By Lusin--Souslin, each inclusion identifies the Borel structure on $H$
with the Borel structure inherited from $K$. The identity homomorphisms
between the two Polish versions of $H$ are therefore Borel. Automatic
continuity makes each identity continuous.
\end{proof}
A subgroup admitting such a topology is often called \emph{Polishable}.
We will specify the complete metrics explicitly, rather than assuming
that the topology inherited from the ambient group is Polish.

\subsection{Circle-valued sequence groups}
Write
\[
 \T=\R/\Z,\qquad
 |t|_{\T}=\min_{a\in\Z}|\widetilde t-a|,
\]
where $\widetilde t$ is any real representative of $t$.
Let $K=\T^\N$ with its compact product topology. Define
\begin{align*}
 H_p&=\left\{h\in K:\sum_n|h_n|_\T^p<\infty\right\},
 &d_p(h,k)&=\left(\sum_n|h_n-k_n|_\T^p\right)^{1/p},\\
 H_0&=\{h\in K:|h_n|_\T\to0\},
 &d_0(h,k)&=\sup_n|h_n-k_n|_\T.
\end{align*}
For a subgroup $H\le K$, its coset relation is
\[
 x\mathrel{K/H}y\quad\Longleftrightarrow\quad x-y\in H.
\]

\begin{lemma}\label{lem:Polish-H}
For $1\le p<\infty$, $(H_p,d_p)$ and $(H_0,d_0)$ are Polish groups,
and their inclusions into $K$ are continuous.
\end{lemma}
\begin{proof}
The triangle inequality on the circle, followed by Minkowski's inequality,
gives the triangle inequality for $d_p$. If $(h^j)$ is $d_p$-Cauchy, it
converges coordinatewise in $K$ to some $h$. Fatou's lemma shows that,
for every sufficiently large $j$, $d_p(h^j,h)$ is small. Applying the
triangle inequality with one fixed $h^j\in H_p$ puts $h$ in $H_p$.
Thus the metric is complete.

A $d_0$-Cauchy sequence has a uniform coordinatewise limit. A uniform
limit of null sequences is null, so $d_0$ is complete as well. In both
cases finite-support sequences with rational circle coordinates form a
countable dense set. Addition and inversion are continuous because the
metrics are invariant, and each coordinate map is continuous.
\end{proof}

\begin{lemma}\label{lem:compact-replacement}
There are continuous reductions $K/H_p\leB \ell_p$ and
$K/H_0\leB c_0$.
\end{lemma}
\begin{proof}
Use the map $b:K\to\R^\N$ that lists the two coordinates
$\cos(2\pi x_n),\sin(2\pi x_n)$ for each $n$. For $s,t\in\T$,
\[
 \left\|(\cos2\pi s,\sin2\pi s)
       -(\cos2\pi t,\sin2\pi t)\right\|_2
 =2\sin\bigl(\pi|s-t|_\T\bigr).
\]
Since $0\le|s-t|_\T\le1/2$,
\[
 4|s-t|_\T\le2\sin\bigl(\pi|s-t|_\T\bigr)
             \le2\pi|s-t|_\T.
\]
The equivalence of the Euclidean norm and the $p$-norm on $\R^2$
gives the summability assertion. The same inequalities give the
null-sequence assertion.
\end{proof}

\subsection{Kernels and quasi-invariance}
Every compact metrizable group has normalized Haar probability,
invariant under translations. A Borel map $x\mapsto\mu_x\in\Prob(Y)$
is a probability kernel. With a probability $\lambda$ on its parameter
space it defines the joint probability
\[
 M(A)=\int\mu_x(A_x)\,d\lambda(x),\qquad
 A_x=\{y:(x,y)\in A\},
\]
which we write $M(dx,dy)=\lambda(dx)\mu_x(dy)$.
The monotone-class theorem gives measurability of the integrand.
We use the elementary density estimates
\[
 \|u^2-v^2\|_1\leq\|u-v\|_2(\|u\|_2+\|v\|_2),
 \qquad |\sqrt a-\sqrt b|^2\leq|a-b|.
\]
Thus convergence of square-root densities in $L^2$ implies
convergence of the measures in total variation.
\subsection{Nonsingular actions and conditional measures}
Suppose a group acts on a standard Borel space $Z$ by Borel bijections.
For a probability $\mu$ and a group element $g$, its pushforward is
$g_*\mu(B)=\mu(g^{-1}B)$. The subgroup
\[
 \QI(\mu)=\{g:g_*\mu\ME\mu\}
\]
is the \emph{quasi-invariance subgroup}. The action is \emph{nonsingular}
if every group element belongs to this subgroup. It is \emph{ergodic}
if every Borel set invariant modulo null sets has measure zero or has
null complement. These definitions also apply to sigma-finite measures.

A \emph{disintegration} of $\mu$ over a Borel map $q:Z\to Y$ to a
standard Borel space is a probability kernel $y\mapsto\nu_y$ such that,
writing $\eta=q_*\mu$, one has
\[
 \mu(B)=\int_Y\nu_y(B)\,d\eta(y),\qquad
 \nu_y(q^{-1}\{y\})=1\quad\text{for almost every }y.
\]
It describes conditional probabilities on the fibers of $q$. The
countable-group ergodic-decomposition theorem used later supplies such
a quotient with ergodic nonsingular conditional measures.

\subsection{Operators and representations}
The standard Borel structure on $B(\mathcal H)$ is generated by the
matrix coefficients relative to a fixed orthonormal basis. On each
norm-bounded ball it agrees with the Borel structures of the usual
strong and weak operator topologies. For $S,T\in B(\mathcal H)$,
\[
 S\UE T\quad\Longleftrightarrow\quad
 S=UTU^*\text{ for some unitary }U.
\]
For a countable discrete group $\Gamma$, $\Rep(\Gamma,\mathcal H)$
is the space of homomorphisms into the unitary group, with pointwise
strong convergence. A unitary equivalence of representations uses the
same unitary for every group element. A unital representation of a
unital $C^*$-algebra is a unital $*$-homomorphism into $B(\mathcal H)$;
its unitary equivalences are defined in the same way.

The operator-algebra facts used later concern semifinite factors.
A \emph{von Neumann algebra} is a unital $*$-algebra of operators
closed in the weak operator topology. Its \emph{commutant} consists
of the operators commuting with all its elements; a \emph{factor}
has only scalar central elements. A faithful normal semifinite trace
$\tau$ assigns a possibly infinite size to positive operators and
determines the Hilbert space $L^2(M,\tau)$ by completing the
square-integrable elements in the norm $\tau(x^*x)^{1/2}$.
Left and right multiplication commute on this space.
For projections $p,q$ of finite trace in a semifinite factor,
the left modules $L^2(M,\tau)p$ and $L^2(M,\tau)q$ are unitarily
equivalent exactly when $\tau(p)=\tau(q)$. The later construction
explains this fact in terms of partial isometries. The crossed-product
and trace facts used there are stated explicitly, with references.

A semifinite factor of type $\mathrm{II}$ has no nonzero minimal
projections. It is of type $\mathrm{II}_1$ when its trace is finite on
the identity and of type $\mathrm{II}_\infty$ when that trace is infinite.
Two representations $\pi,\rho$ of $A$ are \emph{quasi-equivalent} if
there is a normal $*$-isomorphism
$\Phi:\pi(A)''\to\rho(A)''$ with $\Phi(\pi(a))=\rho(a)$ for every
$a\in A$. Unitary equivalence implies quasi-equivalence, but the
converse can fail because the representation's multiplicity, or more
generally its module dimension, is additional data.

A normal representation is one preserving bounded increasing suprema
of positive operators. We use that continuous functions on a compact
metrizable space generate its Borel sigma-algebra; consequently their
multiplication operators generate the full diagonal von Neumann algebra
for any sigma-finite Borel measure. Further crossed-product facts will
be stated at their point of use.

\section{The comparison theorem and established inputs}\label{sec:established}

\begin{theorem}[Complete comparison in the specified family]\label{thm:matrix}
Table~\ref{tab:matrix} and the parameter rules in Section~1 give all
Borel-reducibility comparisons within the stated benchmark family.
\end{theorem}
The proof combines the established results below with the additional
comparisons proved in Sections~\ref{sec:sequence-measure}--\ref{sec:jumps}.
The roadmap for the classical comparisons is Gao's Figure~15.1
\cite[Section~15.2]{Gao}.
Throughout this section, $1\leq r,s<\infty$; the vertex $\ell_p$ in the
matrix denotes a fixed exponent $p>2$.

\subsection{Reductions and operations}

The elementary part of the hierarchy is
\[
 \operatorname{id}(\omega)<_B\operatorname{id}(2^\omega)<_B E_0.
\]
The relation $E_0$ reduces to $E_1$ and $\ell_1$, and
\[
 E_0<_B E_{\infty T}<_B E_\infty.
\]
The strict treeable comparisons are classical
\cite[Section~3.5]{JKL}. In particular, treeability is downward closed
among countable Borel relations \cite[Proposition~3.3]{JKL}, whereas
there are non-treeable countable Borel relations
\cite[Corollary~3.28]{JKL}.
Taking products and using the diagonal maps gives
\[
 E_0^\omega\leq_B E_\infty^\omega,
 \qquad E_0^\omega\leq_B E_1^\omega.
\]
Countable universality and the enumeration of countable Borel classes
also give $E_\infty^\omega\leq_B{=^+}$: each coordinate class is
enumerated and tagged by its coordinate.  These are among the established
arrows summarized by Gao \cite[Figure~15.1]{Gao}.

Dougherty and Hjorth proved the complete finite-exponent comparison
\cite[Theorem~1.1]{DH}:
\begin{equation}\label{eq:dh-survey}
 \ell_r\leq_B\ell_s\quad\Longleftrightarrow\quad r\leq s.
\end{equation}
In particular, $\ell_1<_B\ell_2<_B\ell_p$ for the fixed $p>2$.
These reductions are not supplied by mere inclusion of the corresponding
Banach spaces.  Rosendal's universality theorem for $K_\sigma$
equivalence relations \cite{Rosendal}, in Gao's formulation
\cite[Theorem~8.4.2]{Gao}, gives
\[
 E_1\leq_B\ell_\infty,
 \qquad E_\infty\leq_B\ell_\infty,
 \qquad \ell_r\leq_B\ell_\infty.
\]
Here one uses the standard compact models of the source relations.
Universality for $K_\sigma$ relations does not assert universality for
all $F_\sigma$ relations on arbitrary Polish spaces.

The product-measure construction based on Kakutani's theorem \cite{Kakutani} gives
$\ell_2\leq_B{\equiv_m}$; the atomic-measure construction gives
${=^+}\leq_B{\equiv_m}$
\cite[Section~8.5]{Gao}.
For products and jumps, we use their monotonicity under Borel
reducibility, together with $E\leq_B E^\omega$ and $E\leq_B E^+$.
The tagged-product reduction
\cite[Exercise~15.2.1]{Gao} yields $E_1^\omega\leq_B E_1^+$.
Consequently,
\[
 {=^+}\leq_B E_1^+\leq_B\ell_\infty^+,
 \qquad E_1^\omega\leq_B E_1^+.
\]
The strict jump theorem supplies
$\ell_\infty<_B\ell_\infty^+$
\cite[Theorem~8.3.6]{Gao}.
Theorem~\ref{thm:countable-czero-comparisons} below adds
$E_\infty^\omega\leq_B c_0$, and hence also
$E_\infty\leq_B c_0$.

\subsection{The nonreducibility obstructions}

Several different obstructions are needed; the Borel class of a relation
alone does not determine its position.  Cardinality separates the two
identity relations, and the nonsmoothness of $E_0$ separates it from
$\operatorname{id}(2^\omega)$ \cite{Gao}.

Hjorth's turbulence theory rules out classification of the finite
$\ell_r$ relations and $c_0$ by countable structures
\cite{HjorthBook,Gao}.  In particular, they do not reduce to ${=^+}$ or
$E_\infty$.  Hjorth also established the incomparability of $c_0$ and
each finite $\ell_r$ \cite{HjorthBanach}.  Kanovei and Reeken's
extension of the turbulence obstruction applies to a class closed under
specified product, union, and jump operations and containing $E_1^+$
\cite{KR}.  It gives, in particular,
\[
 \ell_r\nleq_B E_1^+.
\]
We use these published results without reproducing their turbulence
arguments.

A second obstruction comes from \emph{pinnedness}.  Its general
definition concerns whether forcing can create a coherent equivalence
class with no representative in the original universe.  Only the
following established consequences are needed here: ${=^+}$ is not
pinned; pinnedness is inherited by relations Borel reducible to a pinned
relation; orbit relations of Polish groups with a complete compatible
left-invariant metric are pinned; and the relevant product closure
preserves pinnedness \cite{KR,HjorthOrbit}.  The additive Banach groups and
countable products of countable discrete groups satisfy the metric
hypothesis.  Moreover, $E_1$ has $F_\sigma$ classes and is pinned by the
class-complexity criterion of Kanovei and Reeken.  Thus
\begin{equation}\label{eq:pinned-obstructions}
 {=^+}\nleq_B E_\infty^\omega,
 \qquad {=^+}\nleq_B c_0,
 \qquad {=^+}\nleq_B E_1^\omega,
 \qquad {=^+}\nleq_B\ell_r.
\end{equation}
This formulation avoids any use of forcing machinery in the proofs
that follow.

Kechris and Louveau proved that $E_1$ does not reduce to an orbit
relation of a Polish group action \cite[Theorem~4.2]{KL}; see also
\cite[Theorem~10.6.1]{Gao}. Each of $c_0$, $\equiv_m$, and every finite $\ell_r$
is Borel reducible to such an orbit relation, so
\[
 E_1\nleq_B c_0,
 \qquad E_1\nleq_B{\equiv_m},
 \qquad E_1\nleq_B\ell_r.
\]
Gao \cite[Theorem~8.5.2]{Gao} records that $E_0^\omega$
does not reduce to any $\Sigma^0_3$ equivalence relation.
In particular,
\[
 E_0^\omega\nleq_B\ell_\infty
 \quad\hbox{and}\quad
 E_1^\omega\nleq_B\ell_\infty
\]
The second assertion also follows from $E_0^\omega\le_B E_1^\omega$.
Clemens later established stronger primeness statements
\cite{Clemens} (Corollary~3.8 in the author's arXiv version).

Miller's measure-hyperfiniteness theorem
\cite[Theorem~4]{Miller} gives another useful separation:
\begin{equation}\label{eq:miller-separation}
 E_{\infty T}\nleq_B E_1^\omega,\qquad E_\infty\nleq_B E_1^\omega.
\end{equation}
Indeed, a reduction of $E_{\infty T}$ would express the source as a
countable intersection of increasing unions of smooth relations.  Miller's theorem
would make it hyperfinite on a conull Borel set for every probability
measure.  This is incompatible with the nonhyperfinite measured orbit
relation of the free Bernoulli action of the free group on two
generators.  This deduction uses measure hyperfiniteness; it does not
assert that Miller's theorem gives Borel hyperfiniteness. The exact
decomposition and the other treeable comparisons are given in
Section~\ref{sec:treeable}.

Thomas already proved $E_\infty\nleq_B\ell_2$
\cite[Theorem~6.12]{Thomas}, crediting an observation of Kechris and
using Popa's superrigidity theorem.  Equation~\eqref{eq:dh-survey}
therefore gives $E_\infty\nleq_B\ell_r$ for $1\leq r\leq2$.
Theorem~\ref{thm:all-finite-p} below supplies a proof covering every
finite exponent, including $r>2$.

\subsection{Measure equivalence and its applications}

Theorems~\ref{main:lp} and~\ref{main:c0}, proved in
Section~\ref{sec:sequence-measure}, give
\begin{equation}\label{eq:wang-comparisons}
 \ell_r\leq_B{\equiv_m}\Longleftrightarrow r\leq2
 \quad(1\leq r\leq\infty),\qquad c_0\nleq_B{\equiv_m}.
\end{equation}
For $p>2$, $\ell_p$ and $\equiv_m$ are incomparable: the reverse
nonreduction follows from ${=^+}\leq_B{\equiv_m}$ and
\eqref{eq:pinned-obstructions}. Likewise $\equiv_m\nleq_B c_0$,
so $c_0\inc{\equiv_m}$.

\subsection{Completing the comparison matrix}

\begin{proof}[Proof of Theorem~\ref{thm:matrix}]
The remaining jump separations needed are
$c_0\nleq_B\ell_\infty^+$ and
${\equiv_m}\nleq_B\ell_\infty^+$, proved in
Theorems~\ref{thm:c0-jump} and~\ref{thm:measure-jump}.
Together with the results above and
Theorems~\ref{thm:countable-czero-comparisons} and
\ref{thm:treeable-comparisons}, together with
Theorem~\ref{thm:treeable-power-comparisons}, they determine every
entry of the displayed matrix.  The propagation rule is elementary:
if $R\nleq_B S$, $R\leq_B E$, and $F\leq_B S$, then
$E\nleq_B F$.  Thus the matrix records both cited comparisons and
their consequences under composition; it does not claim that every
entry was separately proved in the literature.

\end{proof}

The combinatorial proof of the new reduction to $c_0$ uses the finite
Borel coloring theorem: a Borel graph of maximum degree $d<\infty$
has a Borel coloring with $d+1$ colors
\cite[Proposition~4.6]{KST}.  This existing theorem is used as an
input, rather than reproved.  Kanovei explicitly asked about
$E_\infty$ versus $c_0$ and the finite $\ell_r$ relations
\cite[Question~4.7, p.~28]{Kanovei}; the two jump comparisons are
among the pairs left undecided by Gao's original diagram.

\section{Measure equivalence and sequence spaces}\label{sec:sequence-measure}
\subsection{The classical positive direction}
\begin{proposition}[Dougherty--Hjorth and Kakutani]\label{prop:positive}
For $1\leq p\leq2$, $\ell_p\leq_B{\equiv_m}$.
\end{proposition}
This follows from $\ell_p\leq_B\ell_2$ \cite{DH} and the
product-measure reduction $\ell_2\leq_B{\equiv_m}$ furnished by
Kakutani's theorem \cite{Kakutani}; see \cite[Section~8.5]{Gao}.
We use this classical direction without reproving it.

\subsection{From a Borel reduction to a quasi-invariance subgroup}
The next step is where the definability of a hypothetical reduction meets
the topology of the acting group. It is essential to retain the parameter
as a coordinate of the constructed probability space.

\begin{lemma}[Equivalence of kernels]\label{lem:fibers}
Let $X,Y$ be standard Borel spaces, let $\lambda\in\Prob(X)$, and let
$x\mapsto\mu_x$, $x\mapsto\nu_x$ be probability kernels. For the joint
probabilities
\[
 M(dx,dy)=\lambda(dx)\mu_x(dy),\qquad
 N(dx,dy)=\lambda(dx)\nu_x(dy),
\]
one has
\[
 N\ll M\quad\Longleftrightarrow\quad
 \nu_x\ll\mu_x\ \text{for $\lambda$-almost every $x$}.
\]
Consequently $N\ME M$ if and only if
$\nu_x\ME\mu_x$ for $\lambda$-almost every $x$.
\end{lemma}
\begin{proof}
Suppose first that the stated fiberwise absolute continuity holds. If
$M(A)=0$, nonnegativity gives $\mu_x(A_x)=0$ for almost every $x$.
Hence $\nu_x(A_x)=0$ almost everywhere, and integration gives $N(A)=0$.

Conversely, suppose $N\ll M$, and choose a nonnegative Borel version
$r(x,y)$ of $dN/dM$. Let $\mathcal A$ be a countable generating
$\pi$-system of Borel subsets of $Y$ that contains $Y$. Such a system
exists because a standard Borel space has a countable generating algebra.
For every Borel $B\in\mathcal A$ and every Borel $A\subseteq X$,
\[
 \int_A\nu_x(B)\,d\lambda(x)
 =\int_A\left(\int_B r(x,y)\,d\mu_x(y)\right)d\lambda(x).
\]
Uniqueness of integrable densities on $X$ says that the two integrands
agree almost everywhere. Intersect these conull sets for the countably
many $B\in\mathcal A$. For each $x$ in the resulting conull set,
the measure $r(x,\cdot)\mu_x$ has total mass one, by the equality for
$B=Y$, and agrees with $\nu_x$ on $\mathcal A$. Uniqueness of finite
measures on a generating $\pi$-system gives
$\nu_x=r(x,\cdot)\mu_x$, and therefore $\nu_x\ll\mu_x$.
Applying the assertion in both directions proves the last statement.
\end{proof}

\begin{proposition}[Compact averaging]\label{prop:averaging}
Let $K$ be a compact metrizable abelian group and $H\leq K$ a subgroup.
If $K/H\leB\ME$, then there is a
probability $M$ on the compact metrizable space $Z=K\times2^\N$ such
that, with
\[
 T_g(x,y)=(x+g,y),\qquad M_g=(T_g)_*M,
\]
one has
\begin{equation}\label{eq:quasi-group}
 M_g\ME M\quad\Longleftrightarrow\quad g\in H.
\end{equation}
Moreover, $g\mapsto M_g$ is weakly continuous on $K$.
\end{proposition}
\begin{proof}
Let $x\mapsto\mu_x$ be the assumed Borel reduction, and let $\lambda$
be Haar probability on $K$. Define
\[
 M(dx,dy)=\lambda(dx)\mu_x(dy).
\]
Translation invariance of $\lambda$ gives
\[
 M_g(dx,dy)=\lambda(dx)\mu_{x-g}(dy).
\]
For each fixed $g$, the reduction property says that
$\mu_{x-g}\ME\mu_x$ holds for every $x$ if $g\in H$, and for no $x$
if $g\notin H$. Lemma~\ref{lem:fibers} therefore proves
\eqref{eq:quasi-group}.

If $g_j\to g$ in $K$, then $F(T_{g_j}z)\to F(T_gz)$ for every
$F\in C(Z)$ and every $z\in Z$. Dominated convergence gives
\[
 \int F\,dM_{g_j}=\int F(T_{g_j}z)\,dM(z)
 \longrightarrow\int F(T_gz)\,dM(z)=\int F\,dM_g.
\]
This is weak convergence. Both spaces are metrizable, so the sequential
statement proves continuity.
\end{proof}

The subgroup described by \eqref{eq:quasi-group} is the
\emph{quasi-invariance subgroup} of $M$ for these translations. The
proof does not require a conull set that works simultaneously for all
$g$. It also does not require a measurable choice of sets witnessing
inequivalence of pairs of fiber measures. Finally, weak continuity of
the orbit map comes from the continuous translations of one fixed
measure, although the original kernel need only be Borel.
\subsection{Recovering a compatible Hilbert-space metric}
The topology and unitary representation of quasi-invariance groups
are classical; see Host--M\'ela--Parreau \cite[Sections~5.1 and~5.3]{HMP}
and Gabriyelyan \cite[Theorems~2.1 and~2.2]{Gabri06}.
Related uniform Hilbert-embedding criteria for abelian Polish groups
are discussed in Galindo \cite[Theorems~2.2 and~2.3]{Galindo}.
For the exact compact action obtained above, we give a direct proof
of the complete invariant Hilbert metric that the argument requires. The square-root
density alone need not distinguish two group elements, so we include a
coordinate that remembers the element of the compact group.

\begin{proposition}\label{prop:Hilbert-metric}
Let $H\le K=\T^\N$ carry a Polish group topology whose inclusion into
$K$ is continuous. Suppose a probability $M$ on $Z=K\times2^\N$
satisfies \eqref{eq:quasi-group}. There is an injective map $\Psi$ from
$H$ into a separable real Hilbert space such that
\[
 D(h,k)=\|\Psi(h)-\Psi(k)\|
\]
is a complete translation-invariant metric inducing the given topology
on $H$.
\end{proposition}
\begin{proof}
For $h\in H$, let
\[
 a_h=\sqrt{\frac{dM_h}{dM}}\in L^2(M).
\]
This denotes an equivalence class of functions; no simultaneous
pointwise choice of densities is required. Every $a_h$ is nonnegative
and has norm one. Define
\[
 q(g)=\bigl(2^{-n}\cos(2\pi g_n),
             2^{-n}\sin(2\pi g_n)\bigr)_{n\ge0}\in\ell_2(\N;\R^2).
\]
Uniform convergence of the squared tails shows that $q$ is continuous.
It is injective, and compactness of $K$ makes it a homeomorphism onto
its compact image. Its induced distance is invariant under translations.
Set
\[
 \Psi(h)=(q(h),a_h),\qquad
 D(h,k)^2=\|q(h)-q(k)\|_2^2+\|a_h-a_k\|_{L^2(M)}^2.
\]
The first component makes $D$ a metric. The second is
$\Hel(M_h,M_k)^2$, and a common translation of both measures preserves
Hellinger distance. Thus $D$ is translation invariant. Since $H$ is
abelian, this also gives
\begin{equation}\label{eq:inverse-isometry}
 D(-h,-k)=D(h,k).
\end{equation}

We prove completeness carefully. Let $(h_j)$ be $D$-Cauchy. The compact
image $q(K)$ is closed in its Hilbert space, so there is $g\in K$ with
$q(h_j)\to q(g)$, and hence $h_j\to g$ in $K$. There is also
$a\in L^2(M)$ with $a_{h_j}\to a$. The nonnegative cone is closed
in $L^2(M)$, and the norm is continuous, so $a\ge0$ and $\|a\|_2=1$.
The density-square inequality gives
\[
 \left\|\frac{dM_{h_j}}{dM}-a^2\right\|_{L^1(M)}\longrightarrow0.
\]
Thus $M_{h_j}\to a^2M$ in total variation, and in particular weakly.
The continuous orbit map also gives $M_{h_j}\to M_g$ weakly. Uniqueness
of the weak limit implies
\begin{equation}\label{eq:forward-density}
 M_g=a^2M\ll M.
\end{equation}

One-sided absolute continuity is not yet enough to put $g$ in $H$.
By \eqref{eq:inverse-isometry}, $(-h_j)$ is also $D$-Cauchy, and it
converges to $-g$ in $K$. Repeating the argument gives $M_{-g}\ll M$.
Push this last inclusion forward by $T_g$:
\[
 M=(T_g)_*M_{-g}\ll (T_g)_*M=M_g.
\]
Together with \eqref{eq:forward-density}, this gives $M_g\ME M$,
so $g\in H$. Uniqueness of the nonnegative square root of the density
then gives $a=a_g$. Consequently $D(h_j,g)\to0$, proving completeness.

The Hilbert space used here is separable: $L^2(M)$ is separable for a
finite Borel measure on a compact metrizable space, as follows by
approximating measurable functions using a countable generating algebra.
Every subspace of a separable metric space is separable. Hence $(H,D)$
is a complete separable metric space. Translation invariance and the
triangle inequality give
\[
 D(h+k,h'+k')\le D(h,h')+D(k,k'),
\]
and \eqref{eq:inverse-isometry} gives continuity of inversion. Thus
$(H,D)$ is a Polish group. Its inclusion into $K$ is continuous, because
convergence in $D$ implies convergence of $q(h)$.

Both this topology and the given topology on $H$ are Polish group
topologies with continuous inclusions into $K$. Lemma~\ref{lem:unique-topology}
therefore identifies them. The map $\Psi$ is an isometry for $D$ and
hence a topological embedding for the original topology, as required.
\end{proof}

Combining Propositions~\ref{prop:averaging} and
\ref{prop:Hilbert-metric}, any Borel reduction of $K/H$ to mutual absolute
continuity forces $H$ to admit a compatible invariant Hilbert-space
distance. The next subsection gives a finite-dimensional test for this
conclusion.
\subsection{An elementary Hilbert-space inequality on finite tori}
For an integer $m\ge2$, let $G=(\Z/m\Z)^n$, and let $e_j$ denote its
$j$th coordinate vector. Expectations over $x\in G$ and
$\varepsilon\in\{-1,0,1\}^n$ below are uniform averages. Each sign
vector is mapped into $G$ by reducing its coordinates modulo $m$;
the averaging still counts all $3^n$ sign vectors, even if some have
the same image.

\begin{lemma}\label{lem:torus-inequality}
Let $m\ge2$ and $n\le m^2$. For any real Hilbert space $V$, any map
$F:G\to V$, and any $a_1,\ldots,a_n\in\Z/m\Z$,
\begin{equation}\label{eq:torus}
 \sum_{j=1}^n\E_x\|F(x+a_je_j)-F(x)\|^2
 \le4m^2\E_{x,\varepsilon}\|F(x+\varepsilon)-F(x)\|^2.
\end{equation}
In particular, if $m$ is even, every $a_j$ may be chosen to be $m/2$.
\end{lemma}
\begin{proof}
The range of $F$ lies in a finite-dimensional real subspace. Choosing
an orthonormal basis in that subspace reduces the claim, by summing
coordinatewise, to a scalar-valued function. We may use complex-valued
functions, since the inequality for real functions is then a special case.

The characters of $G$ are
\[
 \chi_k(x)=\exp\left(\frac{2\pi i}{m}\sum_{j=1}^nk_jx_j\right),
 \qquad k\in G.
\]
They form an orthonormal basis for scalar functions under uniform
averaging. This follows by multiplying the one-coordinate identity
\[
 \frac1m\sum_{x=0}^{m-1}\exp(2\pi i kx/m)
 =\begin{cases}1,&k=0\pmod m,\\0,&k\ne0\pmod m,\end{cases}
\]
which is a finite geometric sum. For
$\widehat F(k)=\E_xF(x)\overline{\chi_k(x)}$, Parseval's identity gives
\[
 \E_x|F(x+v)-F(x)|^2
 =\sum_{k\in G}|\widehat F(k)|^2|\chi_k(v)-1|^2.
\]
It therefore suffices to compare the two multipliers for each $k$.

Let $r=|\{j:k_j\ne0\}|$. For any choice of the axis shifts,
\begin{equation}\label{eq:long-multiplier}
 \sum_{j=1}^n|\chi_k(a_je_j)-1|^2\le4r.
\end{equation}
The short-step multiplier is
\begin{equation}\label{eq:short-multiplier}
 \E_\varepsilon|\chi_k(\varepsilon)-1|^2
 =2\left(1-\prod_{j=1}^n q_j\right),\qquad
 q_j=\frac{1+2\cos(2\pi k_j/m)}3.
\end{equation}
Indeed, the signs are independent, each one-coordinate expectation of
the character is the real number $q_j$, and $|z-1|^2=2-2\Re z$ for
$|z|=1$.

If $k_j=0$, then $q_j=1$. If $k_j\ne0$, choose its representative
in $\{1,\ldots,m-1\}$ and put
$s=\min(k_j,m-k_j)$. Then $1\le s\le m/2$. Concavity of sine on
$[0,\pi/2]$ gives $\sin(\pi s/m)\ge2s/m\ge2/m$.
If $q_j\ge0$, we obtain
\[
 1-|q_j|=1-q_j=\frac43\sin^2(\pi k_j/m)
 \ge\frac{16}{3m^2}\ge\frac1{m^2}.
\]
If $q_j<0$, then $|q_j|\le1/3\le1-1/m^2$. In either case,
\[
 |q_j|\le1-\frac1{m^2}\qquad(k_j\ne0).
\]
It follows that
\[
 1-\prod_jq_j\ge1-\prod_j|q_j|
 \ge1-(1-m^{-2})^r.
\]
For completeness, if $u=m^{-2}$ and $0\le r\le m^2$, then
\[
 (1-u)^{-r}\ge(1+u)^r\ge1+ru,
\]
so
\[
 1-(1-u)^r\ge\frac{ru}{1+ru}\ge\frac{ru}{2}
 =\frac{r}{2m^2}.
\]
The case $r=0$ is included. Thus the multiplier in
\eqref{eq:short-multiplier} is at least $r/m^2$. Multiplication by
$4m^2$ dominates \eqref{eq:long-multiplier}, proving the result.
\end{proof}

This estimate is a Hilbert-space instance of metric cotype; compare
\cite[Proposition 3.1]{MN}. The proof above supplies the precise scaling
and constant used here. In particular the dimension may be $n=m^2$.

\begin{proposition}\label{prop:lp-obstruction}
If $2<p<\infty$, the Polish group $(H_p,d_p)$ admits no topological embedding
$\Psi$ into a real Hilbert space whose induced distance
$D(h,k)=\|\Psi(h)-\Psi(k)\|$ is translation invariant.
\end{proposition}
\begin{proof}
Suppose such a map exists. Continuity of its inverse at $\Psi(0)$
supplies $c>0$ such that
\begin{equation}\label{eq:long-threshold}
 d_p(h,0)\ge\frac14\quad\Longrightarrow\quad D(h,0)\ge c.
\end{equation}
Continuity of $\Psi$ at zero then supplies $\delta>0$ such that
\begin{equation}\label{eq:short-threshold}
 d_p(h,0)<\delta\quad\Longrightarrow\quad D(h,0)<\frac c3.
\end{equation}
Both constants are fixed before choosing a dimension.

Because $2/p-1<0$, choose an even integer $m\ge2$ so large that
$m^{2/p-1}<\delta$, and put $n=m^2$. Embed $G=(\Z/m\Z)^n$ in $H_p$
by
\[
 \iota(x)_j=\begin{cases}x_j/m\pmod1,&1\le j\le n,\\0,&j=0\text{ or }j>n.
 \end{cases}
\]
The formula is independent of the integer representatives of $x_j$,
and defines a group homomorphism. Let $F=\Psi\circ\iota$.
An axis half-turn satisfies
\[
 d_p\bigl(\iota((m/2)e_j),0\bigr)=\frac12.
\]
Translation invariance and \eqref{eq:long-threshold} therefore give
$\|F(x+(m/2)e_j)-F(x)\|\ge c$ for every $x,j$. Every short sign
increment satisfies
\[
 d_p\bigl(\iota(\varepsilon),0\bigr)
 \le\frac{n^{1/p}}m=m^{2/p-1}<\delta.
\]
Thus \eqref{eq:short-threshold} and translation invariance give
$\|F(x+\varepsilon)-F(x)\|<c/3$ for every $x,\varepsilon$.
Applying Lemma~\ref{lem:torus-inequality} yields
\[
 m^2c^2=nc^2
 \le\sum_{j=1}^n\E_x\|F(x+(m/2)e_j)-F(x)\|^2
 \le4m^2\E_{x,\varepsilon}\|F(x+\varepsilon)-F(x)\|^2
 \le\frac49m^2c^2,
\]
a contradiction.
\end{proof}

The endpoint $p=2$ is excluded precisely because the short-step bound
in this scaling is then $m^{2/p-1}=1$, which does not approach zero.
No rate of continuity of the hypothetical embedding is assumed.

\begin{theorem}\label{thm:lp}
For $1\le p\leq\infty$,
\[
 \ell_p\leB\ME\quad\Longleftrightarrow\quad p\le2.
\]
\end{theorem}
\begin{proof}
Proposition~\ref{prop:positive} proves the forward construction for
$1\le p\le2$. If $2<p<\infty$ and $\ell_p\leB\ME$, composition with
Lemma~\ref{lem:compact-replacement} gives $K/H_p\leB\ME$.
Propositions~\ref{prop:averaging} and \ref{prop:Hilbert-metric} then
produce an invariant Hilbert-space metric compatible with $d_p$.
This contradicts Proposition~\ref{prop:lp-obstruction}.

For $p=\infty$, the classical reductions and nonreductions
$E_1\leq_B\ell_\infty$ and $E_1\nleq_B{\equiv_m}$
from Section~\ref{sec:established} imply
$\ell_\infty\nleq_B{\equiv_m}$.
This proves Theorem~\ref{main:lp} in its full stated range.
\end{proof}

\subsection{The null-sequence relation}
\begin{proposition}\label{prop:c0-obstruction}
The Polish group $(H_0,d_0)$ admits no topological embedding into a real
Hilbert space for which the induced distance is translation invariant.
\end{proposition}
\begin{proof}
Suppose $\Psi$ were such an embedding, and write
$D(h,k)=\|\Psi(h)-\Psi(k)\|$. As in the proof of
Proposition~\ref{prop:lp-obstruction}, choose $c>0$ and then $\delta>0$
such that
\[
 d_0(h,0)\ge\tfrac14\ \Longrightarrow\ D(h,0)\ge c,
 \qquad
 d_0(h,0)<\delta\ \Longrightarrow\ D(h,0)<c/3.
\]
Choose an even integer $m\ge2$ with $1/m<\delta$, set $n=m^2$, and
use the same finite-support grid $\iota:(\Z/m\Z)^n\to H_0$.
Every axis half-turn has $d_0$-distance $1/2$ from zero, whereas every
short sign increment has distance at most $1/m$. Translation invariance
gives the long- and short-step estimates at every point of the grid.
Lemma~\ref{lem:torus-inequality}, applied to $F=\Psi\circ\iota$, gives
\[
 m^2c^2\le 4m^2(c/3)^2,
\]
which is impossible.
\end{proof}

\begin{theorem}\label{thm:c0}
\[
 c_0\nleB\ME.
\]
\end{theorem}
\begin{proof}
A proposed reduction, composed with the continuous map in
Lemma~\ref{lem:compact-replacement}, would reduce $K/H_0$ to $\ME$.
Compact averaging and Proposition~\ref{prop:Hilbert-metric} would then
give the embedding excluded by Proposition~\ref{prop:c0-obstruction}.
This also proves Theorem~\ref{main:c0}.
\end{proof}

The short-step estimate for $H_0$ is independent of the dimension:
$d_0(\iota(\varepsilon),0)\le1/m$. This is the precise geometric reason
that the same finite-torus argument applies to null sequences. No
comparison of $c_0$ with an $\ell_p$ relation is used.

\section{Countable relations and sequence spaces}\label{sec:countable-sequences}
\subsection{Countable equivalence relations and \texorpdfstring{$c_0$}{c0}}
We give the positive answer to Kanovei's $c_0$ question.
For a bounded-degree Borel graph $G$, let $G^{[r]}$ connect distinct
points at graph distance at most $r$. This is again a bounded-degree
Borel graph. By \cite[Proposition~4.6]{KST}, each $G^{[r]}$ has a
finite Borel proper coloring. We use this theorem without repeating
its proof.
\begin{theorem}\label{thm:countable-graph-czero}
Let $G$ be a Borel graph of bounded finite degree on a nonempty standard
Borel space $X$, and let $E_G$ be its connectedness relation. There is a
Borel reduction of $E_G$ to $c_0$ with range contained in $[0,1]^{\N}$.
\end{theorem}

\begin{proof}
For every $n\geq1$, choose a finite Borel proper coloring $q_n$ of
$G^{[2n]}$. Fix a Borel injection $\iota:X\to 2^{\N}$ and let
$\mathcal P_n$ be the finite partition determined by
\[
 \bigl(\iota(x)\mathbin{\upharpoonright}n,q_1(x),\ldots,q_n(x)\bigr).
\]
These partitions are nested and separate points. Moreover, distinct
points in one $\mathcal P_n$-cell have graph-distance greater than $2n$.

For every cell $C\in\mathcal P_n$, define
\begin{equation}\label{eq:countable-distance-block}
 f_{n,C}(x)=\max\left\{0,1-\frac{d_G(x,C)}{n}\right\},
 \qquad d_G(x,C)=\inf_{z\in C}d_G(x,z),
\end{equation}
with the value $0$ when the distance is infinite. The coordinate
$f_{n,C}$ is Borel: bounded graph neighborhoods of a Borel set are Borel,
and this function takes values in the finite set
$\{0,1/n,\ldots,1\}$. Concatenate the finite, nonempty blocks
$(f_{n,C}(x))_{C\in\mathcal P_n}$, in increasing order of $n$, to obtain
a Borel map $f:X\to[0,1]^{\N}$.

Suppose first that $x\mathrel{E_G}y$, and put $D=d_G(x,y)$. Within a
connected component, distance to a set is $1$-Lipschitz whenever finite.
If the component misses $C$, both values in
\eqref{eq:countable-distance-block} are zero. Thus, in every case,
\[
 |f_{n,C}(x)-f_{n,C}(y)|\leq\min\{1,D/n\}.
\]
The bound is uniform over the $n$th block. Since each block is finite,
$f(x)-f(y)$ tends to zero.

Conversely, suppose $f(x)-f(y)\in c_0$. There is $N$ such that
\[
 |f_{n,C}(x)-f_{n,C}(y)|<\tfrac14
 \quad(n\geq N,\ C\in\mathcal P_n).
\]
Let $C_n$ be the cell containing $x$. Since $f_{n,C_n}(x)=1$, there is
$z_n\in C_n$ with $d_G(y,z_n)<n/4$. A finite graph-distance infimum is
attained because the possible distances are nonnegative integers.

Nestedness puts both $z_n$ and $z_{n+1}$ in $C_n$, while
\[
 d_G(z_n,z_{n+1})
 \leq d_G(z_n,y)+d_G(y,z_{n+1})
 <\frac n4+\frac{n+1}{4}<2n.
\]
The coloring property forces $z_n=z_{n+1}$. Hence the witnesses stabilize
to one point $z$ belonging to $C_n$ for all $n\geq N$. Since the
partitions separate points, $z=x$. But $z$ lies in the component of $y$,
so $x\mathrel{E_G}y$.
\end{proof}

\begin{corollary}\label{thm:countable-czero-comparisons}
We have $E_\infty\le_B c_0$ and $E_\infty^\omega\le_B c_0$.
\end{corollary}
\begin{proof}
The full $\mathbb F_2$ shift admits a Borel graphing of degree at most
four: join each point to its translates by the two generators and
their inverses, omitting loops and repeated edges.
Theorem~\ref{thm:countable-graph-czero} applies.
The product conclusion then follows from
$(c_0)^\omega\sim_B c_0$, the established product identity
recorded in \cite[Lemma~7.3]{GaoBanach}.
\end{proof}

\subsection{The universal countable relation versus every finite \texorpdfstring{$\ell_p$}{ell p}}
\label{sec:finite-p}

We now settle the comparison left open by the restriction to
$1\le p\le2$ in Thomas's theorem. The obstruction below concerns
the natural coset relation $\ell_p$, not arbitrary actions of
the additive Banach group.

\begin{theorem}[Spectral-gap homomorphisms]\label{thm:spectral-homomorphism}
Let a countable group $\Gamma$ act on a standard probability space
$(X,\mu)$ by probability-preserving Borel transformations, and suppose
the action has spectral gap. Fix $1\le p<\infty$.
If a measurable map $f:X\to\R^\N$ satisfies
\[
 f(gx)-f(x)\in\ell_p
 \qquad\text{for every }g\in\Gamma
 \text{ and almost every }x,
\]
then there is $m\in\R^\N$ such that
$f(x)-m\in\ell_p$ for almost every $x$.
\end{theorem}

\begin{proof}
Choose a finite median $m_n$ of each real random variable $f_n$,
and set $m=(m_n)_{n\in\N}$.
Use \eqref{eq:scalar-poincare} for this exponent, and write $S$
for its finite set and $C_p$ for its constant.
For $N\ge1$ put
\[
 v_N(x)=(f_0(x)-m_0,\ldots,f_{N-1}(x)-m_{N-1})\in\ell_p^N.
\]
For $R>0$ define the radial retraction onto the ball of radius $R$ by
\[
 r_R(v)=
 \begin{cases}
 v\min\{1,R/\|v\|_p\},&v\ne0,\\
 0,&v=0,
 \end{cases}
 \qquad u_{N,R}=r_R\circ v_N .
\]
Each coordinate of $u_{N,R}$ is bounded and has median zero.
Indeed, its multiplier is strictly positive, so its positive
and negative sets are exactly those of the corresponding coordinate
of $v_N$.
Apply \eqref{eq:scalar-poincare} separately to these $N$ coordinates
and add. This gives
\begin{equation}\label{eq:radial-poincare}
 \int\min\{\|v_N(x)\|_p^p,R^p\}\,d\mu(x)
 \le C_p\sum_{s\in S}
      \int\|u_{N,R}(sx)-u_{N,R}(x)\|_p^p\,d\mu(x).
\end{equation}

The radial retraction in any normed space is $2$-Lipschitz and
has image of diameter at most $2R$. Thus
\begin{equation}\label{eq:radial-bound}
 \|r_R(v)-r_R(w)\|
      \le2\min\{\|v-w\|,R\}.
\end{equation}
To check the Lipschitz bound, suppose $\|v\|\ge\|w\|$ and write
$a=\min\{1,R/\|v\|\}$, $b=\min\{1,R/\|w\|\}$, with the
natural value $1$ at zero. Then $a\le b$ and
\[
 \|av-bw\|\le a\|v-w\|+(b-a)\|w\|
              \le2\|v-w\|.
\]
The last bound follows by considering whether $v,w$ are inside
or outside the ball.

For $s\in S$ let
$c_s(x)=\|f(sx)-f(x)\|_p$, which is measurable and finite
almost everywhere. No assertion that $c_s^p$ is integrable is made.
Combining \eqref{eq:radial-poincare} and
\eqref{eq:radial-bound} yields
\begin{equation}\label{eq:truncated-energy}
 \int\min\left\{\sum_{n<N}|f_n(x)-m_n|^p,R^p\right\}\,d\mu(x)
 \le 2^p C_p\sum_{s\in S}
           \int\min\{c_s(x)^p,R^p\}\,d\mu(x).
\end{equation}
All integrands in this inequality are bounded for fixed $R$.

First let $N\to\infty$, using monotone convergence. Put
$H(x)=\sum_n|f_n(x)-m_n|^p\in[0,\infty]$ and divide by $R^p$:
\[
 \int\min\{H(x)/R^p,1\}\,d\mu(x)
 \le 2^pC_p\sum_{s\in S}
       \int\min\{c_s(x)^p/R^p,1\}\,d\mu(x).
\]
Now let $R\to\infty$. Bounded convergence sends the left side
to $\mu(H=\infty)$, while every term on the right tends to zero,
since $c_s$ is finite almost everywhere and $S$ is finite.
It follows that $H<\infty$ almost everywhere. This is exactly
the required conclusion $f(x)-m\in\ell_p$.
\end{proof}

\begin{theorem}\label{thm:all-finite-p}
For every $1\le p<\infty$, the relations $E_\infty$ and $\ell_p$
are Borel incomparable:
\[
 E_\infty\nleq_B\ell_p
 \qquad\text{and}\qquad
 \ell_p\nleq_B E_\infty .
\]
\end{theorem}

\begin{proof}
Equip $X=2^{\mathbb F_2}$ with the fair Bernoulli product
probability measure. Its shift action preserves this atomless
measure and has spectral gap \cite[Remark~6.4]{KidaTuckerDrob}.
Suppose $F:X\to\R^\N$ reduced its orbit relation to $\ell_p$.
It is a homomorphism, so Theorem~\ref{thm:spectral-homomorphism}
gives an $m\in\R^\N$ with
$F(x)-m\in\ell_p$ on a conull set $A$.
Consequently all points of $A$ have $\ell_p$-equivalent images.
The reverse implication in the definition of a reduction forces
$A$ to lie in one orbit. Every orbit is countable and has measure
zero, a contradiction.
This orbit relation is a model of $E_\infty$; alternatively its
reducibility to $E_\infty$ already suffices.

The reverse nonreduction is the established turbulence obstruction
for $\ell_p$ recalled in Section~\ref{sec:established}
\cite{HjorthBook,Gao}. It is cited here, not reproved.
\end{proof}

\begin{remark}
The argument works for every finite $p$, but does not extend by
letting $p=\infty$: its central step is the coordinate sum in
\eqref{eq:truncated-energy}. In fact
$E_\infty\le_B\ell_\infty$ by Rosendal's theorem.
Nor does it give an obstruction for $c_0$, to which $E_\infty$
does reduce. The order of the two limits in the proof is essential:
the number of coordinates tends to infinity before the truncation
radius does.
\end{remark}

\subsection{The universal treeable relation}\label{sec:treeable}

We determine the additional row and column of Table~\ref{tab:matrix}.
The treeable model from Section~\ref{sec:prelim} is essential for the
finite-exponent obstruction: a nonreduction with source $E_\infty$
cannot in general be restricted to an arbitrary subrelation.

\begin{theorem}\label{thm:treeable-comparisons}
Within the benchmark family of Section~1, the relations strictly
below $E_{\infty T}$ are
\[
 \id(\N),\quad\id(2^\N),\quad E_0.
\]
The relations strictly above $E_{\infty T}$ are
\[
 E_{\infty T}^\omega,\quad E_\infty,\quad E_\infty^\omega,\quad =^+,\quad
 c_0,\quad\meq,\quad\ell_\infty,\quad E_1^+,\quad\ell_\infty^+.
\]
The incomparable relations are
\[
 E_0^\omega,\quad E_1,\quad E_1^\omega,\quad
 \ell_p\quad(1\le p<\infty).
\]
\end{theorem}
\begin{proof}
Write $T=E_{\infty T}$ on $X_T=\operatorname{Free}(2^{\mathbb F_2})$.
Give the full shift its fair Bernoulli probability $\mu$.
The set $X_T$ is conull: for each nonidentity $g\in\mathbb F_2$,
the condition $g\cdot x=x$ forces equal independent bits along an
infinite $\langle g\rangle$-orbit and has probability zero; then take
the countable union over $g$. The restricted action preserves
probability and retains the spectral gap of the Bernoulli action
\cite[Remark~6.4]{KidaTuckerDrob}. Its classes are countable and
$\mu$-null.

Suppose that $f:X_T\to\R^\N$ reduced $T$ to $\ell_p$ for some
finite $p$. Theorem~\ref{thm:spectral-homomorphism} gives
$m\in\R^\N$ such that $f(x)-m\in\ell_p$ on a conull set.
The reduction condition forces that set into one $T$-class,
contradicting its measure zero. Hence $T\nleq_B\ell_p$.
Conversely, $\ell_p\nleq_B T$ because turbulence implies that
$\ell_p$ is not essentially countable \cite{HjorthBook,Gao}.

Next suppose that $f=(f_j)_{j\in\N}$ reduced $T$ to $E_1^\omega$.
For $j,k\in\N$, define
\[
 x\,R_{j,k}\,y
 \quad\Longleftrightarrow\quad
 (\forall n\ge k)\ (f_j(x))_n=(f_j(y))_n.
\]
The tail-code map makes each $R_{j,k}$ smooth, and
$R_{j,k}\subseteq R_{j,k+1}$. Moreover,
\[
 T=\bigcap_{j\in\N}\bigcup_{k\in\N}R_{j,k}.
\]
Miller's theorem \cite[Theorem~4]{Miller} implies that $T$ is
hyperfinite on a $\mu$-conull Borel set. This contradicts the
nonhyperfiniteness of the free probability-preserving Bernoulli
$\mathbb F_2$-action; see \cite[Section~3.5]{JKL}.
Thus $T\nleq_B E_1^\omega$, and the same holds for the smaller
targets $E_1$ and $E_0^\omega$.
Kechris--Louveau's orbit-relation obstruction \cite{KL} gives
$E_1\nleq_B T$, and hence $E_1^\omega\nleq_B T$.
Finally, $E_0^\omega\nleq_B T$ follows from
\cite[Theorem~8.5.2]{Gao}, since the presentation of $T$ above is
$F_\sigma$. Alternatively, the classical Sixth Dichotomy
\cite[Theorem~7.1]{HK2001}, together with the nonessential countability
of $E_0^\omega$, gives its incomparability with the nonhyperfinite
countable relation $T$.

The relations below $T$ follow from the classical chain
\[
 \id(\N)<_B\id(2^\N)<_B E_0<_B T<_B E_\infty
\]
recalled in Section~\ref{sec:established}.
The diagonal map gives $T\leq_B T^\omega$. This inequality is
strict because $E_0^\omega\leq_B T^\omega$ and $E_0^\omega$ is not
essentially countable, whereas $T$ is countable. Each of the other
claimed upper relations lies above $E_\infty$:
\[
\begin{gathered}
 E_\infty\leq_B E_\infty^\omega\leq_B{=^+}\leq_B{\equiv_m},
 \qquad E_\infty\leq_B c_0,\\
 {=^+}\leq_B E_1^+\leq_B\ell_\infty^+,
 \qquad E_\infty\leq_B\ell_\infty.
\end{gathered}
\]
Here the reduction to $c_0$ is
Theorem~\ref{thm:countable-czero-comparisons}; it can also be applied
directly to the degree-four treeing of $T$.
These remaining upper comparisons are strict because $E_\infty\nleq_B T$.
This accounts for every comparison involving $T$.
\end{proof}

\begin{remark}[Coset relations and arbitrary group actions]
Allison proved that every countable treeable Borel equivalence relation
is Borel reducible to the orbit relation of some continuous action of
the additive Polish group $\ell_1$ \cite[Theorem~1.2]{Allison}.
This answers negatively Hjorth's question whether every countable
Borel relation classifiable by an abelian Polish group is hyperfinite
\cite[Question~1.1]{Allison}. The theorem concerns the existence of
an action; it does not identify that action with translation on
$\R^\N$. The latter action has Gao's coset relation $\ell_1$,
to which Theorem~\ref{thm:treeable-comparisons} gives a nonreduction.
Thus the treeable conclusion here is consistent with Allison's
positive classification theorem. Frisch--Shinko subsequently announced
the corresponding arbitrary-action classification for every countable
Borel relation \cite[slide~8]{FrischShinkoSlides}, and their slides
also treat countable products \cite[slide~19]{FrischShinkoSlides}.
The question about the natural $\ell_1$ coset relation is explicitly
distinguished in \cite[slide~20]{FrischShinkoSlides}, where the
conjecture that countable relations below it must be hyperfinite is
attributed to Hjorth. Our nonreductions concern particular source
relations and do not settle that general conjecture.
We use the cited treeability and measure-hyperfiniteness results with
attribution; the additional
sequence-space comparisons are consequences of the constructions
and spectral-gap theorem proved in this paper.
\end{remark}

\subsection{The countable power of the universal treeable relation}
\label{sec:treeable-power}

Put $P=E_{\infty T}^\omega$. The following theorem records its place
in the enlarged benchmark list. The strict interval and the
incomparability with $E_\infty$ are consequences of the classical
results cited in Section~\ref{sec:property-T}; their inclusion here
adds a comparison vertex, not a claim of a new rigidity theorem.

\begin{theorem}\label{thm:treeable-power-comparisons}
The relations in our family strictly below $P$ are
\[
 \id(\N),\quad\id(2^\N),\quad E_0,\quad E_0^\omega,\quad E_{\infty T}.
\]
The relations strictly above $P$ are
\[
 E_\infty^\omega,\quad =^+,\quad c_0,\quad\meq,\quad
 E_1^+,\quad\ell_\infty^+.
\]
The incomparable relations are
\[
 E_\infty,\quad E_1,\quad E_1^\omega,\quad\ell_\infty,\quad
 \ell_p\quad(1\le p<\infty).
\]
In particular,
\[
 E_0^\omega<_B P<_B E_\infty^\omega,
 \qquad P\inc E_\infty.
\]
\end{theorem}
\begin{proof}
Let $T=E_{\infty T}$. Coordinatewise reductions give
$E_0^\omega\leq_B P\leq_B E_\infty^\omega$, and the diagonal
map gives $T\leq_B P$. If $P\leq_B E_0^\omega$, then
$T\leq_B E_0^\omega$, contradicting
Theorem~\ref{thm:treeable-comparisons}. Hence the first inequality
is strict. Since $E_0^\omega$ is not essentially countable,
$P\nleq_B E_\infty$ and $T<_B P$ as well.

To prove the crucial reverse obstruction, take
$\Gamma=\mathrm{SL}_3(\Z)$ acting by shifts on $2^\Gamma$ with
fair Bernoulli probability $\mu$, and let $R$ be its orbit relation.
The measure is atomless and invariant. The action is ergodic:
finite-coordinate cylinder sets become independent once the translated
coordinate supports are disjoint, and measurable sets are approximated
by cylinder sets. Theorem~\ref{thm:propertyT-treeable}, with
$F=T$, gives $R\nleq_B P$. As $R\leq_B E_\infty$, we obtain
\[
 E_\infty\nleq_B P.
\]
Consequently $E_\infty^\omega\nleq_B P$, proving the upper
strict inequality and the asserted incomparability with $E_\infty$.
This is the application of the classical countable-product obstruction
from \cite[Section~3.7, after Theorem~3.31]{JKL}.

All the listed upper targets lie above $E_\infty^\omega$, using
Theorem~\ref{thm:countable-czero-comparisons} for $c_0$ and the
established reductions of Section~\ref{sec:established} for the others.
They therefore lie above $P$. Each also lies above $E_\infty$,
so none reduces to $P$. This proves all the upper strictness assertions.
The remaining lower assertions follow by composition from
$\id(\N)<_B\id(2^\N)<_B E_0<_B T<_B P$.

For the tail relations, $T\leq_B P$ and
$T\nleq_B E_1^\omega$ imply $P\nleq_B E_1^\omega$ and
$P\nleq_B E_1$. Conversely, $P$ is an orbit relation of the
Polish group $(\mathbb F_2)^\N$, so the Kechris--Louveau theorem
\cite{KL} gives $E_1\nleq_B P$. It follows that
$E_1^\omega\nleq_B P$ and $\ell_\infty\nleq_B P$, since both
lie above $E_1$. The nonreduction $P\nleq_B\ell_\infty$ follows
from $E_0^\omega\leq_B P$ and
$E_0^\omega\nleq_B\ell_\infty$.

Finally, $T\leq_B P$ and $T\nleq_B\ell_p$ give
$P\nleq_B\ell_p$ for every finite $p$. In the other direction,
$P\leq_B E_\infty^\omega\leq_B{=^+}$, whereas
$\ell_p\nleq_B{=^+}$ by turbulence \cite{HjorthBook,Gao}.
Thus $\ell_p\nleq_B P$. Every new matrix entry is accounted for.
\end{proof}

\begin{remark}
The presentation of $P$ in Section~\ref{sec:prelim} is $\Pi^0_3$.
It has no Borel-isomorphic $\Sigma^0_3$ presentation: otherwise
$E_0^\omega\leq_B P$ would contradict
\cite[Theorem~8.5.2]{Gao}. Its universality for countable products of
treeable countable Borel relations and its distinct position in
Theorem~\ref{thm:treeable-power-comparisons} make it a natural
additional benchmark.
\end{remark}

\section{The jump obstructions}\label{sec:jumps}
\subsection{The jump of bounded difference and \texorpdfstring{$c_0$}{c0}}\label{sec:c0-jump}
For an abelian group $K$ and a subgroup $H$, the notation $K/H$ below
means the coset equivalence relation $x\sim y$ if $x-y\in H$.
\begin{lemma}[From locally constant rows to a constant countable set]
\label{lem:rows}
Let $E$ be a Borel equivalence relation on a nonempty Polish space $X$.
Suppose every $E$-invariant Baire set is meager or comeager. Let
$F:X\to(\mathbb R^{\mathbb N})^{\mathbb N}$ be a Borel homomorphism
from $E$ to $\ell_\infty^+$, and write $F_k(x)=F(x)_k$.
Assume that, for every $k$ and every nonempty open $A\subseteq X$, there
are a nonempty open $O\subseteq A$ and $a\in\mathbb R^{\mathbb N}$ such
that $F_k(x)\mathrel{\ell_\infty}a$ for comeager many $x\in O$.
Then $F$ is constant modulo $\ell_\infty^+$ on a comeager subset of $X$.
\end{lemma}

\begin{proof}
Fix $k$. Apply the assumption inside every member of a countable base for
$X$. We obtain countably many open sets $O_{k,j}$ with dense union and
representatives $a_{k,j}$ such that the exceptional subset of each
$O_{k,j}$ is meager. Remove these exceptional sets and the complement of
their dense open union. Repeating this for all $k$ and taking a countable
intersection gives a comeager set $X_0\subseteq X$ and one countable
family $(a_r)_{r\in\mathbb N}$ such that
\[
  x\in X_0\quad\Longrightarrow\quad
  (\forall k)(\exists r)\ F_k(x)\mathrel{\ell_\infty}a_r.
\]
Define the Borel map $b:X\to2^{\mathbb N}$ by
\[
  b(x)(r)=1\quad\Longleftrightarrow\quad
  (\exists k)\ F_k(x)\mathrel{\ell_\infty}a_r.
\]
The homomorphism property makes every coordinate of $b$ invariant under
$E$ on all of $X$. By generic ergodicity, each coordinate is constant on
a comeager set. Thus $b$ is constant on a single comeager set $X_1$.
For $x,y\in X_0\cap X_1$, all row classes belong to the fixed countable
family, and equality of $b(x)$ and $b(y)$ says exactly that the sets of
row classes coincide. Hence $F(x)\mathrel{\ell_\infty^+}F(y)$.
\end{proof}

\begin{theorem}
\label{thm:c0-jump}
The equivalence relation $c_0$ is not Borel reducible to
$\ell_\infty^+$.
\end{theorem}

\begin{proof}
Write $\mathbb T=\mathbb R/\mathbb Z$, give it its shortest-arc metric
$d_{\mathbb T}$ of diameter $1/2$, and put
\[
 K=\mathbb T^{\mathbb N},\qquad
 H=\{h\in K:d_{\mathbb T}(h_n,0)\longrightarrow0\}.
\]
With the metric
$\rho(h,h')=\sup_n d_{\mathbb T}(h_n,h'_n)$, the group $H$ is Polish:
uniform Cauchy sequences converge uniformly to null sequences, and
finite-support rational sequences form a countable dense subset.
It acts continuously on $K$ by translation. Its orbits are dense and
meager. Density follows from finite-support translations. For
meagerness, $H$ is contained in the meager set
\[
 \bigcup_{N\in\mathbb N}\bigcap_{n\ge N}
 \{h\in K:d_{\mathbb T}(h_n,0)\le1/8\},
\]
whose displayed closed constituents have empty interior.

Every $H$-invariant Baire subset of $K$ is meager or comeager. Indeed, if
such a set is comeager in a nonempty open set $U$, the translates of $U$
under the countable group of finite-support rational elements cover $K$;
the corresponding translates of the meager exceptional set show that
the invariant set is comeager in $K$.

Suppose $F:K\to(\mathbb R^{\mathbb N})^{\mathbb N}$ is a Borel
homomorphism from $K/H$ to $\ell_\infty^+$. We verify the local row
condition of Lemma~\ref{lem:rows}. Fix a row $k$ and a nonempty open
$A\subseteq K$. On the nonempty open set
\[
 \Omega_A=\{(x,h)\in K\times H:x+h\in A\}
\]
the sets
\[
 B_{l,n}=\{(x,h):
       \sup_i|F_k(x+h)_i-F_l(x)_i|\le n\},
 \qquad l,n\in\mathbb N,
\]
form a countable Borel cover. The Baire theorem therefore supplies
$l,n$ and a nonempty open rectangle $U\times W\subseteq\Omega_A$ on
which $B_{l,n}$ is comeager. Choose a dense $G_\delta$ set $D\subseteq K$
such that $F_k|D$ is continuous.

In $U\times W\times W$, a comeager set of triples satisfies both
$ (x,h_1)\in B_{l,n}$ and $(x,h_2)\in B_{l,n}$, as well as
$x+h_1,x+h_2\in D$. The endpoint requirements preserve comeagerness
because the maps $(x,h_1,h_2)\mapsto x+h_j$ are continuous and open.
For all such triples,
\begin{equation}
\label{eq:compact-good-triples}
 \sup_i|F_k(x+h_1)_i-F_k(x+h_2)_i|\le2n.
\end{equation}

Fix $h_0\in W$ and $0<\eta<1/4$ with
$h_0+\{h\in H:\rho(h,0)<\eta\}\subseteq W$.
We claim that
\begin{equation}
\label{eq:compact-small-pairs}
 \sup_i|F_k(y)_i-F_k(z)_i|\le2n
\end{equation}
whenever $y,z\in D\cap(U+h_0)$ and
$\sup_i d_{\mathbb T}(y_i,z_i)<\eta$.
To see this, set $x=y-h_0$ and $t=z-y$. Let $t^{(r)}$ agree with $t$ on
its first $r$ coordinates and vanish elsewhere. Then
$(x,h_0,h_0+t^{(r)})\in U\times W\times W$, and its two translated
endpoints converge in $K^2$ to $(y,z)$. Approximate each of these triples
by a triple from the comeager set used in
\eqref{eq:compact-good-triples}, closely enough that its endpoints still
converge to $(y,z)$. All approximating endpoints lie in $D$.
Continuity of $F_k|D$ and closedness of the uniform bound in
\eqref{eq:compact-good-triples} prove the claim.

Choose a basic cylinder $O\subseteq U+h_0$ whose finitely many restricted
coordinates are open arcs of length less than $\min(\eta,1/2)$.
Choose an integer $m$ with $1/(2m)<\eta$. Let $Q\subseteq O^2$ consist
of pairs $(y,z)$ having no antipodal coordinate pair. It is dense
$G_\delta$ in $O^2$. On $Q$ use the principal arguments in
$(-1/2,1/2)$ to define
\[
 P_j(y,z)=y+\frac jm\operatorname{arg}(z-y),
 \qquad j=0,\ldots,m,
\]
coordinatewise. All $P_j(y,z)$ belong to $O$, and consecutive terms are
at sup circle distance at most $1/(2m)<\eta$.

Each map $P_j:Q\to O$ is continuous, open, and onto. For a restricted
coordinate this is ordinary affine interpolation in a short open
interval. For an unrestricted coordinate the principal-arc map is
locally affine away from antipodal pairs, hence open, and is onto by
taking equal endpoints. The product map is open because a basic product
neighborhood restricts only finitely many coordinates and every other
coordinate map is onto. It follows that, for comeager pairs in $O^2$,
all of the finitely many points $P_j(y,z)$ lie in $D$.
Applying \eqref{eq:compact-small-pairs} successively gives
\[
 \sup_i|F_k(y)_i-F_k(z)_i|\le2nm
\]
for comeager pairs in $O^2$. By Kuratowski--Ulam there is a fixed
$a\in\mathbb R^{\mathbb N}$ such that
$F_k(y)\mathrel{\ell_\infty}a$ for comeager $y\in O$.
Since $O\subseteq A$, this proves the required local row condition.

Lemma~\ref{lem:rows} now makes $F$ constant modulo $\ell_\infty^+$ on a
comeager subset of $K$. As every $H$-orbit is meager, $F$ cannot be a
reduction. Finally, the continuous chord map
\[
 x\longmapsto
 (\cos(2\pi x_0),\sin(2\pi x_0),
  \cos(2\pi x_1),\sin(2\pi x_1),\ldots)
\]
reduces $K/H$ to Gao's real null-difference relation $c_0$.
A reduction $c_0\le_B\ell_\infty^+$ would therefore give the impossible
reduction $K/H\le_B\ell_\infty^+$.
\end{proof}

\begin{remark}
The compact source in this proof is essential to the argument. For the
usual action of $c_0$ on $\mathbb R^{\mathbb N}$, the identity map is a
homomorphism to $\ell_\infty$ with no comeager constant fiber. We have
made no generic-ergodicity assertion for that different topological
presentation.
\end{remark}
\subsection{Measure equivalence and the jump of \texorpdfstring{$\ell_\infty$}{ell infinity}}
\label{sec:measure-jump}
Let $X$ be the space of atomless full-support probabilities on $C=2^\N$.
Here full support means that every nonempty open set has positive measure.
This is a dense $G_\delta$ subspace of $\Prob(C)$, hence is Polish.

\begin{theorem}\label{thm:measure-jump}
Every Borel homomorphism
$F:X\to(\R^\N)^\N$ from $\meq$ to $\ell_\infty^+$ is constant modulo
$\ell_\infty^+$ on a comeager subset of $X$. Consequently
$\meq\nleq_B\ell_\infty^+$.
\end{theorem}

The proof uses square-root densities. We first provide their topology
explicitly so that every use of Baire category has a Polish domain.

\subsubsection{A Polish space of positive densities}
For $\mu\in X$ let $(b_i^\mu)_{i\in\N}$ be its real clopen Haar basis.
The first function is $b_0^\mu=1$. For each finite binary string $s$ there
is a further basis function supported on its cylinder $[s]$, with values
\[
 \sqrt{\frac{\mu[s1]}{\mu[s0]\mu[s]}}\quad\hbox{on }[s0],
 \qquad
 -\sqrt{\frac{\mu[s0]}{\mu[s1]\mu[s]}}\quad\hbox{on }[s1].
\]
Enumerate the strings by increasing length. These functions form an
orthonormal basis of $L^2(\mu;\R)$: their finite initial spans include all
clopen simple functions. Their coefficients depend continuously on
finitely many positive cylinder masses.
For $a\in\ell_2$ write
$f_{\mu,a}=\sum_i a_i b_i^\mu$ in $L^2(\mu)$, and set
\[
 Z=\{(\mu,a)\in X\times\ell_2:\ \|a\|_2=1,\
                         f_{\mu,a}>0\ \mu\hbox{-almost everywhere}\}.
\]
We sometimes denote this point by $(\mu,f)$.
Define its source and target by
\[
 s(\mu,f)=\mu,\qquad t(\mu,f)=f^2\mu.
\]
In particular $s(z)\meq t(z)$.

\begin{lemma}\label{lem:density-bundle}
The space $Z$ is Polish. The maps $s,t:Z\to X$ are continuous open
surjections.
\end{lemma}
\begin{proof}
Nonnegativity of $f_{\mu,a}$ is equivalent to
$\int_A f_{\mu,a}\,d\mu\ge0$ for every clopen set $A$.
Each of these integrals is continuous: $1_A$ belongs to a fixed finite
initial Haar span. Thus nonnegativity is a closed condition.
On this closed set strict positivity almost everywhere is equivalent to
\[
 \forall r\ge1\ \exists k\ge1:\quad
 \int_C \min(1,k f_{\mu,a})\,d\mu>1-\frac1r.
\]
The integral here is continuous. Use the globally $k$-Lipschitz function
\[
 \phi_k(u)=\min(1,\max(0,ku)).
\]
Truncate the Haar
expansion, and bound the integral of the error by
$k\|a-a^{(N)}\|_2$. The tail bound is locally uniform in the coefficient
vector. For a finite expansion the integral is a finite sum over
clopen cells, hence continuous. Strict positivity is therefore
$G_\delta$. Together with the closed norm-one condition this proves that
$Z$ is Polish.

The source map is continuous. To see that it is open, fix an open
neighborhood $W$ of $(\mu_0,f_0)$ in $Z$. Strictly positive clopen simple
functions, normalized in $L^2(\mu_0)$, approximate $f_0$ in $L^2(\mu_0)$.
Choose one, $w$, whose point over $\mu_0$ belongs to $W$. The section
\begin{equation}\label{eq:measure-star}
\mu\longmapsto
 \left(\mu,\frac{w}{\|w\|_{L^2(\mu)}}\right)
\end{equation}
is continuous in Haar coefficients: it has only finitely many nonzero
coefficients, all continuous in $\mu$. Consequently $s(W)$ contains
a neighborhood of $\mu_0$. This argument applies over every point of
$s(W)$.

For a continuous function $g$ on $C$, the integral
$\int g f_{\mu,a}^2\,d\mu$ is continuous. Finite Haar truncations prove
the assertion first for simple $f$, and the estimate
\[
 \left|\int g(f^2-h^2)\,d\mu\right|
 \le \|g\|_\infty\,\|f-h\|_2(\|f\|_2+\|h\|_2)
\]
proves the general assertion. This is continuity of $t$; its values lie
in $X$ because equivalence preserves atomlessness and full support.

Finally inversion is continuous. If $\nu=f^2\mu$, the inverse point is
$(\nu,1/f)$. Its $j$th Haar coefficient is
\[
 \int b_j^\nu\,\frac1f\,d\nu=\int b_j^\nu f\,d\mu.
\]
It is continuous by the already proved continuity of $t$ and of the
integrals against clopen simple functions. Coordinatewise convergence
of these coefficient vectors implies norm convergence, since their
norms and the norm of the limit are all one. Inversion is an
involution, so it is a homeomorphism. Since $t=s\circ I$, the target
map is open. Both maps are onto by taking $f=1$.
\end{proof}

\subsubsection{A bound on one output row}
Fix a Borel homomorphism $F$ as in Theorem~\ref{thm:measure-jump},
an output row $k$, and a nonempty open set $A\subseteq X$.
Write $F_k(\mu)\in\R^\N$ for the $k$th row.
For every $z\in Z$ there are $l,n\in\N$ such that
\begin{equation}\label{eq:measure-1}
\bigl(F_k(tz),F_l(sz)\bigr)\in L_n.
\end{equation}
These countably many Borel conditions cover the nonempty open space
$t^{-1}(A)$. Some fixed condition is comeager on a nonempty open set
$W\subseteq t^{-1}(A)$.

The source fiber square
\[
 W\times_s W=\{(z_1,z_2)\in W^2:s(z_1)=s(z_2)\}
\]
is Polish. Its projections onto $W$ are open surjections: the first
projection of $(U_1\times U_2)\cap(W\times_s W)$ is
$U_1\cap s^{-1}(s(U_2))$, which is open.
Thus on a comeager subset of this fiber square,
\begin{equation}\label{eq:measure-2}
\bigl(F_k(tz_1),F_k(tz_2)\bigr)\in L_{2n}.
\end{equation}
Choose a comeager set $D\subseteq X$ on which $F_k$ is continuous.
Openness of the projections and of $t$ lets us require additionally
$tz_1,tz_2\in D$. The resulting good set is still comeager, and in
particular dense in the fiber square.

\subsubsection{Closure and Hellinger distance}
Use the Hellinger and total variation conventions of Section~2. Concavity of the square root gives
\begin{equation}\label{eq:measure-3}
\Hel\left(\mu,\frac{\mu+\nu}{2}\right)^2
       \le \frac12 \Hel(\mu,\nu)^2.
\end{equation}

Choose a point of $W$ represented by a strictly positive clopen simple
function $w$ over $\mu_0$, normalized by $\int w^2\,d\mu_0=1$, and put
$\nu_0=w^2\mu_0$. Such points exist by the section approximation above.
There are an open neighborhood $U$ of $\mu_0$ and $\epsilon>0$ such that
\begin{equation}\label{eq:measure-4}
\rho\in U,\quad
 \left\|f-\frac{w}{\|w\|_{L^2(\rho)}}\right\|_{L^2(\rho)}<\epsilon
 \quad\Longrightarrow\quad(\rho,f)\in W
\end{equation}
for normalized strictly positive $f$.
The map
\[
 R_w(\sigma)=\frac{w^{-2}\sigma}{\int w^{-2}\,d\sigma}
\]
is weakly continuous and sends $\nu_0$ to $\mu_0$.
Choose a convex basic weak-open neighborhood $O$ of $\nu_0$, contained
in $A$, such that $R_w(O)\subseteq U$.

We claim that
\begin{equation}\label{eq:measure-5}
\mu_1,\mu_2\in O\cap D,\ \Hel(\mu_1,\mu_2)<\epsilon
 \quad\Longrightarrow\quad
       (F_k(\mu_1),F_k(\mu_2))\in L_{2n}.
\end{equation}
Set $\sigma=(\mu_1+\mu_2)/2$ and $\rho=R_w(\sigma)$.
Writing $w_\rho=w/\|w\|_{L^2(\rho)}$, the normalization gives
$w_\rho^2\rho=\sigma$.
If $\mu_1\meq\mu_2$, the normalized positive densities
$f_j=\sqrt{d\mu_j/d\rho}$ satisfy
\[
 \|f_j-w_\rho\|_{L^2(\rho)}
 =\Hel(\mu_j,\sigma)
 \le 2^{-1/2}\Hel(\mu_1,\mu_2)<\epsilon.
\]
Hence $(\rho,f_1),(\rho,f_2)$ belong to $W$.

For arbitrary $\mu_1,\mu_2$, replace $\mu_j$ by
$\mu_j^{(r)}=(1-r)\mu_j+r\sigma$, with $0<r<1$.
Both replacements are equivalent to $\sigma$, remain in $O$, and
converge weakly to $\mu_j$ as $r\downarrow0$.
By the same concavity argument,
$\Hel(\mu_j^{(r)},\sigma)\le \Hel(\mu_j,\sigma)$.
Thus they are targets of a point of $W\times_s W$.
Approximate each such point by points of the dense good set from \eqref{eq:measure-2},
with target errors tending to zero. The resulting target pairs lie in
$D^2$ and converge to $(\mu_1,\mu_2)$.
Relative continuity of $F_k$ on $D$ and closedness of $L_{2n}$ prove \eqref{eq:measure-5}.
No intersection of a weak comeager set with a fixed Hellinger ball
has been assumed.

\subsubsection{Interpolation and generic constancy}
The map
\[
 P_q:X^2\to X,\qquad P_q(\mu,\nu)=(1-q)\mu+q\nu
\]
is continuous and open for every $q\in[0,1]$.
At $q=0,1$ this is a projection. For $0<q<1$, take source
neighborhoods of $\mu,\nu$ specified by masses on a common finite
clopen partition. All cell masses are positive. Put
$\lambda=(1-q)\mu+q\nu$.
For $\lambda'\in X$ sufficiently close to $\lambda$ in these masses,
set, on every partition cell $B$,
\[
 \mu'(B)=\mu(B),\qquad
 \nu'(B)=\frac{\lambda'(B)-(1-q)\mu(B)}q>0.
\]
Within $B$, take $\mu'$ and $\nu'$ to be the corresponding constant
multiples of $\lambda'|B$. They are atomless full-support probabilities
in the prescribed neighborhoods, and
$(1-q)\mu'+q\nu'=\lambda'$. This proves openness.
The restriction $P_q:O^2\to O$ is also open and onto, by convexity and
the diagonal identity.

Choose an integer $m$ with $2/m<\epsilon^2$.
For $\mu,\nu\in O$ put $\mu_j=P_{j/m}(\mu,\nu)$.
Then
\[
 \Hel(\mu_j,\mu_{j+1})^2
 \le \|\mu_j-\mu_{j+1}\|_{\rm TV}
 \le 2/m<\epsilon^2.
\]
Openness implies that for comeager pairs $(\mu,\nu)\in O^2$, all
$\mu_0,\ldots,\mu_m$ belong to $D$.
Apply \eqref{eq:measure-5} to consecutive points and add the bounds to obtain
\[
 (F_k(\mu),F_k(\nu))\in L_{2mn}
\]
for comeager pairs. Kuratowski--Ulam now implies that $F_k$ is constant
modulo $\ell_\infty$ on a comeager subset of $O$.
This proves the local row property required by Lemma~\ref{lem:rows},
since $k$ and $A$ were arbitrary.

Measure equivalence on $X$ is generically ergodic.
Indeed, the additive Polish group $C(C,\R)$ acts continuously by
\[
 (h,\mu)\longmapsto \frac{e^h\mu}{\int e^h\,d\mu}.
\]
Its orbits are contained in $\meq$-classes and are weakly dense:
a strictly positive clopen simple density can match the masses of
any finite clopen partition. The criterion in the preliminaries applies.
Lemma~\ref{lem:rows} therefore gives generic constancy of $F$ modulo
$\ell_\infty^+$.

Every $\meq$-class in $X$ is meager. Fix $\mu\in X$.
Probabilities singular to $\mu$ form a dense $G_\delta$ subset of
$\Prob(C)$. The $G_\delta$ description follows from the clopen
witnesses
\[
 \forall r\ge1\ \exists B\text{ clopen}:\quad
       \mu(B)<2^{-r},\qquad \nu(B)>1-2^{-r}.
\]
These witnesses characterize singularity: a summable sequence of such
errors gives a $\mu$-null, $\nu$-conull liminf set; the converse follows
by regularity and clopen approximation.
Density follows because finite atomic probabilities are dense and are
singular to the atomless measure $\mu$.
Intersecting with the dense $G_\delta$ space $X$ gives a comeager subset
of $X$ disjoint from the $\meq$-class of $\mu$.
A reduction cannot map a comeager set into one target class, because
its preimage would lie in one meager source class.
This finishes the proof of Theorem~\ref{thm:measure-jump}.

\section{Unitary classification and measure equivalence}\label{sec:operators}
\subsection{The compact group and its two subgroups}\label{sec:groups}
For the rest of the proof, $K$ denotes the following compact group,
rather than the product of circles used for the sequence-space obstruction.
Let
\[
 L=\widehat{\Q_{\mathrm{discrete}}}
   =\{\chi\in\T^\Q:\chi(r+s)=\chi(r)+\chi(s)\},
 \qquad K=L^\N.
\]
Thus $L$ is a closed subgroup of a countable product of circles and is
compact metrizable. As before, $\T=\R/\Z$ is written additively. Thus a character here
is a homomorphism into the additive circle. Define
\[
 \iota(t)(r)=rt+\Z,\qquad
 j:\R^\N\longrightarrow K,\quad j(x)_n=\iota(x_n).
\]
The map $j$ is a continuous injective homomorphism. Injectivity follows
because $rt\in\Z$ for every rational $r$ forces $t=0$. Its image is
Borel and $j$ has a Borel inverse on that image. No claim that it is
a topological embedding is needed.

Put
\begin{gather*}
 B_0=\left\{h\in\ell_1(\R):\sum_n h_n=0\right\},
 \qquad H=j(B_0),\qquad G=j(\ell_1(\R)),\\
 D=j\left(\left\{d\in\Q^{(\N)}:\sum_n d_n=0\right\}\right).
\end{gather*}
Here $\Q^{(\N)}$ denotes the finitely supported rational sequences.
Both $G$ and $H$ are Borel: their Banach-space domains map continuously
and injectively into $K$. Give $H$ the Polish group topology transported
from the $\ell_1$ norm on the closed Banach subspace $B_0$.
Its inclusion in $K$ is continuous and $D$ is countable and dense for
this finer topology. Density follows by finite truncation, correcting
the truncated sum in one coordinate, and rational approximation with
one final rational correction.

Index coordinates by $\N=\{0,1,2,\ldots\}$. For $a\in\R$ and $n\ge1$,
set
\begin{equation}\label{eq:vn}
 v_n(a)=j(ae_0-ae_n),\qquad g(a)=j(ae_0).
\end{equation}
The facts used throughout the proof are
\begin{equation}\label{eq:group-facts}
 v_n(a)\in H,\quad \|v_n(a)\|_H=2|a|,\quad
 v_n(a)\longrightarrow g(a)\text{ in }K,\quad
 mg(a/m)=g(a).
\end{equation}
Moreover $G=H+g(\R)$ and $g(1)\notin H$.

\subsection{A rigidity property of quasi-invariance groups}
\label{sec:measure-obstruction}

Throughout this subsection, $K$, $j$, $H$, $G$, and $D$ are as defined
above. Thus
\[
 B_0=\left\{h\in\ell_1(\R):\sum_{n=0}^{\infty}h_n=0\right\},
 \qquad H=j(B_0),\qquad G=j(\ell_1(\R)),
\]
and $D$ is the image under $j$ of the finitely supported rational
sequences of sum zero. We give $B_0$ its $\ell_1$ norm and identify $H$
with $B_0$ when discussing this finer Polish group topology. In
particular, $D$ is dense in $H$ for this topology. Recall also that
\[
 v_n(a)=j(a(e_0-e_n))\longrightarrow g(a)=j(ae_0)
 \quad\text{in }K\qquad(n\longrightarrow\infty),
\]
where $n\geq1$ and the $B_0$ norm of $a(e_0-e_n)$ is $2|a|$.

If $K$ acts continuously on a compact metrizable space $Z$ and
$\mu\in\Prob(Z)$, write
\[
 k\mu=(z\mapsto k\cdot z)_*\mu,
 \qquad
 \QI(\mu)=\{k\in K:k\mu\text{ is equivalent to }\mu\}.
\]
The quasi-invariance group is a subgroup of $K$. No singularity
assumption is imposed on translates outside this subgroup.

\begin{theorem}\label{thm:rigidity}
For every continuous action of $K$ on a compact metrizable space $Z$
and every $\mu\in\Prob(Z)$,
\[
 H\subseteq\QI(\mu)\quad\Longrightarrow\quad G\subseteq\QI(\mu).
\]
\end{theorem}

We first establish the measure-theoretic ingredients of the proof.

\subsubsection{Affinity and conditional measures}

For two probability measures $\rho,\sigma$ on a standard Borel space,
their Hellinger affinity is
\[
 \Aff(\rho,\sigma)
 =\int\sqrt{\frac{d\rho}{d\lambda}
                  \frac{d\sigma}{d\lambda}}\,d\lambda,
\]
where $\lambda$ is any finite measure dominating both. The expression
is independent of $\lambda$, takes values in $[0,1]$, and is zero
exactly when $\rho$ and $\sigma$ are mutually singular. With the
normalization of Hellinger distance used above,
$\Hel(\rho,\sigma)^2=2-2\Aff(\rho,\sigma)$.

\begin{lemma}\label{lem:affinity-usc}
If $Z$ is compact metrizable, then
\begin{equation}\label{eq:affinity-variational}
 \Aff(\rho,\sigma)
 =\inf_{\substack{f\in C(Z)\\ f>0}}
       \frac12\left(\int f\,d\rho+\int f^{-1}\,d\sigma\right).
\end{equation}
Consequently affinity is jointly upper semicontinuous, and hence
Borel, on $\Prob(Z)^2$ with its weak topology.
\end{lemma}

\begin{proof}
Set $\lambda=\rho+\sigma$, $p=d\rho/d\lambda$, and
$q=d\sigma/d\lambda$, choosing Borel versions with $p,q\geq0$ and
$p+q=1$ almost everywhere. The arithmetic--geometric mean inequality
gives
\[
 \sqrt{pq}\leq\frac12(fp+f^{-1}q)
\]
for every positive measurable $f$. This proves one inequality in
\eqref{eq:affinity-variational}.

For $0<\delta\leq1$, let
\[
 f_\delta=\sqrt{\frac{q+\delta}{p+\delta}}.
\]
Its values lie in a compact interval $[c_\delta,C_\delta]\subset(0,\infty)$.
Moreover,
\[
 \frac12(pf_\delta+qf_\delta^{-1})\longrightarrow\sqrt{pq}
 \quad\text{almost everywhere as }\delta\downarrow0.
\]
Both summands are bounded by $\sqrt2$: for example,
$pf_\delta\leq\sqrt{p(q+\delta)}\leq\sqrt2$.
Dominated convergence therefore applies. For each fixed $\delta$,
regularity of the finite Borel measure $\lambda$ gives continuous
functions converging to $f_\delta$ in $L^1(\lambda)$. Clipping these
functions to $[c_\delta,C_\delta]$ preserves continuity and convergence.
Inversion is Lipschitz on this interval, so their reciprocals also
converge in $L^1(\lambda)$. Thus the infimum over positive continuous
functions is at most
$\frac12\int(pf_\delta+qf_\delta^{-1})\,d\lambda$.
Letting $\delta\downarrow0$ proves the formula. Each expression in its
infimum is weakly continuous in $(\rho,\sigma)$, proving upper
semicontinuity.
\end{proof}

\begin{lemma}\label{lem:affinity-disintegration}
Let $q:Z\to Y$ be a Borel map between standard Borel spaces. Suppose
$\rho,\sigma\in\Prob(Z)$ have the same marginal $\eta$ under $q$, and
write their disintegrations as
\[
 \rho=\int_Y\rho_y\,d\eta(y),\qquad
 \sigma=\int_Y\sigma_y\,d\eta(y),
\]
where the conditional probabilities are concentrated on $q^{-1}\{y\}$
for almost every $y$. Then
\begin{equation}\label{eq:conditional-affinity}
 \Aff(\rho,\sigma)
   =\int_Y\Aff(\rho_y,\sigma_y)\,d\eta(y).
\end{equation}
Furthermore, $\rho$ and $\sigma$ are equivalent if and only if
$\rho_y$ and $\sigma_y$ are equivalent for $\eta$-almost every $y$.
\end{lemma}

\begin{proof}
Let $\lambda=(\rho+\sigma)/2$ and
$\lambda_y=(\rho_y+\sigma_y)/2$. These are a measure and its
disintegration over $q$, with marginal $\eta$. Put
$r=d\rho/d\lambda$ and $s=d\sigma/d\lambda$.
For almost every $y$, the restrictions of $r$ and $s$ are respectively
the densities of $\rho_y$ and $\sigma_y$ relative to $\lambda_y$.
Here is the usual uniqueness argument: for every Borel $C\subseteq Y$
and every member $B$ of a countable generating algebra on $Z$,
\[
 \int_C\int_B r\,d\lambda_y\,d\eta(y)
 =\rho(B\cap q^{-1}(C))
 =\int_C\rho_y(B)\,d\eta(y).
\]
Include $B=Z$, discard one null set for the countable algebra, and
use uniqueness of finite measures to obtain the assertion for all
Borel $B$. The same argument applies to $s$.
Fubini's theorem applied to $\sqrt{rs}$ now proves
\eqref{eq:conditional-affinity}.

For the last assertion, $\rho\ll\sigma$ is equivalent to
$\lambda(\{r>0,\ s=0\})=0$. Disintegrating this nonnegative quantity
shows that it is equivalent to $\rho_y\ll\sigma_y$ almost everywhere.
Apply the same argument in the other direction.
\end{proof}

\begin{lemma}\label{lem:ergodic-dichotomy}
Two ergodic quasi-invariant probability measures for the same Borel
action of a countable group are either equivalent or mutually singular.
They need not have the same Radon--Nikodym cocycle.
\end{lemma}

\begin{proof}
Let the measures be $\rho,\sigma$ and set
$\lambda=(\rho+\sigma)/2$. This measure is also quasi-invariant.
Let $A$ and $B$ be the positive-density sets of $\rho$ and $\sigma$
relative to $\lambda$. Both are invariant modulo $\lambda$: a group
element sends the essential support of either measure to the
essential support of its translate, which has the same measure class.

If $\lambda(A\cap B)=0$, the measures are singular. Otherwise both
$\rho(A\cap B)$ and $\sigma(A\cap B)$ are positive. Ergodicity of
$\rho$ applied to $B$ gives $\rho(B)=1$, and ergodicity of $\sigma$
applied to $A$ gives $\sigma(A)=1$. Hence $A=B$ modulo $\lambda$,
and the two measures are equivalent.
\end{proof}

\subsubsection{Continuity and the invariant quotient}

\begin{lemma}\label{lem:koopman-continuity}
Suppose $H\subseteq\QI(\mu)$. The nonsingular Koopman representation
of $B_0$ on $L^2(Z,\mu)$,
\[
 (U_h\xi)(z)
   =\sqrt{r_h(z)}\,\xi(j(-h)\cdot z),
 \qquad r_h=\frac{d(j(h)\mu)}{d\mu},
\]
is strongly continuous for the $\ell_1$ topology on $B_0$.
In particular,
\[
 \Aff(\mu,j(h)\mu)=\langle 1,U_h1\rangle\longrightarrow1
 \quad\text{as }\|h\|_1\longrightarrow0.
\]
\end{lemma}

\begin{proof}
We first verify measurability; this is necessary because a Borel scalar
positive definite function alone need not be continuous. Choose
increasing finite Borel partitions $\mathcal P_n$ whose union generates
the Borel sigma-algebra of $Z$. The functions
\[
 r_{h,n}(z)=
 \sum_{\substack{A\in\mathcal P_n\\\mu(A)>0}}
       \frac{(j(h)\mu)(A)}{\mu(A)}\,1_A(z)
\]
are jointly Borel in $(h,z)$. Indeed,
$(j(h)\mu)(A)=\int 1_A(j(h)\cdot z)\,d\mu(z)$ is Borel in $h$.
For each fixed $h$, these functions are the conditional expectations
of $d(j(h)\mu)/d\mu$. Martingale convergence therefore gives a jointly
Borel version of $r_h$: take their finite positive limit where it
exists, and assign the value one elsewhere. For each $h$, the
exceptional set has $\mu$-measure zero.

The change-of-variables formula and the Radon--Nikodym chain rule show
that the displayed operators form a unitary representation. Its matrix
coefficients between bounded Borel functions are Borel functions of
$h$, by joint measurability and integration. Such functions contain a
countable dense subset of $L^2(Z,\mu)$. Since this Hilbert space is
separable, its unitary group with the strong operator topology is
Polish, and these matrix coefficients show that $h\mapsto U_h$ is a
Borel homomorphism into that group.

Automatic continuity of Borel homomorphisms between Polish groups,
recalled in Section~\ref{sec:prelim}, applies to $h\mapsto U_h$.
The final identity follows directly from the definition of affinity.
\end{proof}

\begin{lemma}\label{lem:fixed-quotient}
Assume $H\subseteq\QI(\mu)$. There are a standard Borel space $Y$, a
Borel map $q:Z\to Y$, a probability $\eta=q_*\mu$, and a measurable
family $(\nu_y)_{y\in Y}$ such that
\[
 \mu=\int_Y\nu_y\,d\eta(y),
\]
almost every $\nu_y$ is $D$-ergodic and $D$-quasi-invariant and is
concentrated on $q^{-1}\{y\}$, and, for each fixed $h\in H$,
\begin{equation}\label{eq:quotient-fixed}
 q(h\cdot z)=q(z)\quad\text{for }\mu\text{-almost every }z.
\end{equation}
For each such $h$,
\begin{equation}\label{eq:translate-disintegration}
 \Aff(\mu,h\mu)
   =\int_Y\Aff(\nu_y,h\nu_y)\,d\eta(y).
\end{equation}
\end{lemma}

\begin{proof}
The nonsingular ergodic decomposition theorem for a countable group
provides a countably generated invariant quotient with the asserted
conditional measures; see \cite[Theorem~1.5]{GS}, particularly its
invariance, conditional-expectation, and fiber-concentration assertions.
Thus $q$ is $D$-invariant modulo $\mu$ and can be taken to generate
the invariant sigma-algebra modulo null sets. We use the disintegration
over this quotient, not merely a barycentric representation by ergodic
measures.

To prove \eqref{eq:quotient-fixed}, let $f$ be any bounded
$D$-invariant measurable function,
and let $M_f$ be multiplication by $f$ on $L^2(Z,\mu)$. The Koopman
formula gives
\[
 U_hM_fU_h^*=M_{f\circ(j(-h)\cdot)}\qquad(h\in B_0).
\]
In particular, $M_f$ commutes with $U_d$ for every $d\in D$, using
the identification of $D$ with its preimage in $B_0$. By density of
$D$ and Lemma~\ref{lem:koopman-continuity}, it commutes with every
$U_h$, $h\in B_0$. Hence $f$ is invariant modulo $\mu$ under every
element of $H$.

Apply this to the indicators of the inverse images under $q$ of a
countable family of Borel subsets separating points of $Y$. For each
fixed $h$, intersecting the corresponding countably many conull sets
gives \eqref{eq:quotient-fixed}. There is no assertion of a common
conull set for all $h$.

Fix $h\in H$. Disintegrating the null exceptional set in
\eqref{eq:quotient-fixed} shows
that $h\nu_y$ is concentrated on $q^{-1}\{y\}$ for almost every $y$.
Consequently
\[
 h\mu=\int_Y h\nu_y\,d\eta(y)
\]
is the disintegration of $h\mu$ over the same quotient, with the same
base probability $\eta$. Lemma~\ref{lem:affinity-disintegration} gives
\eqref{eq:translate-disintegration}.
\end{proof}

\subsubsection{The roots argument}

\begin{proof}[Proof of Theorem~\ref{thm:rigidity}]
Assume $H\subseteq\QI(\mu)$ and fix the disintegration in
Lemma~\ref{lem:fixed-quotient}. For $s\geq0$, put
\[
 \varepsilon(s)=
 \sup_{\substack{h\in B_0\\\|h\|_1\leq2s}}
       \bigl(1-\Aff(\mu,j(h)\mu)\bigr).
\]
Then $0\leq\varepsilon(s)\leq1$, and
$\varepsilon(s)\to0$ as $s\downarrow0$ by
Lemma~\ref{lem:koopman-continuity}.

Fix $a\in\R$. Since $\|a(e_0-e_n)\|_1=2|a|$,
\eqref{eq:translate-disintegration} gives
\begin{equation}\label{eq:uniform-affinity-bound}
 \int_Y\Aff(\nu_y,v_n(a)\nu_y)\,d\eta(y)
 =\Aff(\mu,v_n(a)\mu)\geq1-\varepsilon(|a|).
\end{equation}
For every fixed probability $\nu_y$, continuity of the $K$ action and
compactness of $Z$ imply weak convergence
$v_n(a)\nu_y\to g(a)\nu_y$. Lemma~\ref{lem:affinity-usc} therefore
gives
\[
 \Aff(\nu_y,g(a)\nu_y)
 \geq\limsup_n\Aff(\nu_y,v_n(a)\nu_y).
\]
All these quantities lie in $[0,1]$. The reverse Fatou inequality,
together with \eqref{eq:uniform-affinity-bound}, yields
\begin{equation}\label{eq:limit-affinity-bound}
 \int_Y\Aff(\nu_y,g(a)\nu_y)\,d\eta(y)
 \geq1-\varepsilon(|a|).
\end{equation}

Restrict to a fixed conull Borel subset $Y_0\subseteq Y$ on which
$\nu_y$ is $D$-ergodic and $D$-quasi-invariant. Since $K$ is abelian,
$g(a)$ commutes with $D$, so $g(a)\nu_y$ has these two properties as
well. Lemma~\ref{lem:ergodic-dichotomy} shows that these two measures
are equivalent or singular. Thus the Borel set
\[
 B_a=\{y\in Y_0:g(a)\notin\QI(\nu_y)\}
     =\{y\in Y_0:\Aff(\nu_y,g(a)\nu_y)=0\}
\]
satisfies, by \eqref{eq:limit-affinity-bound},
\[
 \eta(B_a)\leq\varepsilon(|a|).
\]
For every positive integer $m$, the identity $m g(a/m)=g(a)$ and the
subgroup property of $\QI(\nu_y)$ imply
$B_a\subseteq B_{a/m}$. Hence
\[
 \eta(B_a)\leq\varepsilon(|a|/m)\longrightarrow0.
\]
Only countably many roots are used for this fixed $a$. Any exceptional
sets arising in \eqref{eq:translate-disintegration} can likewise be
discarded simultaneously for the
countable family $v_n(a/m)$. No choice of a conull set uniform in all
real $a$ is required.

It follows that $g(a)\nu_y$ and $\nu_y$ are equivalent for almost
every $y$. Integrating their null sets in both directions gives
$g(a)\mu\sim\mu$. Since $a$ was arbitrary, every $g(a)$ belongs to
$\QI(\mu)$. Finally, for $h\in\ell_1(\R)$, set
$a=\sum_nh_n$. The decomposition
\[
 h=(h-ae_0)+ae_0,
 \qquad h-ae_0\in B_0,
\]
and the subgroup property give $j(h)\in\QI(\mu)$. This proves
$G\subseteq\QI(\mu)$.
\end{proof}

\subsection{The compact coset relation does not reduce to measure equivalence}
\label{sec:coset-nonreduction}

Let $K/H$ be the coset equivalence relation on $K$:
\[
 x\mathrel{K/H}y\quad\Longleftrightarrow\quad x-y\in H.
\]
We take $\ME$ on $\Prob(2^{\N})$, where it denotes mutual absolute
continuity. This is the usual standard Borel model of measure
equivalence; changing the underlying uncountable standard Borel space
does not change its Borel isomorphism type.

\begin{theorem}\label{thm:nonreduction}
The relation $K/H$ is not Borel reducible to $\ME$:
\[
 K/H\nleB\ME.
\]
\end{theorem}

\begin{proof}
If $K/H\leB{\equiv_m}$, Proposition~\ref{prop:averaging} supplies a
probability $M$ on $K\times2^\N$ whose exact quasi-invariance group
for first-coordinate translation is $H$. Theorem~\ref{thm:rigidity}
would then give $G\subseteq H$. This is impossible because
$j(e_0)\in G\setminus H$.
\end{proof}

\begin{remark}
The compact-averaging argument uses only nonequivalence of the reduction
kernel values $\mu_{x-k}$ and $\mu_x$ when $k\notin H$. It does not
assume, or infer, that those measures
are mutually singular. The ergodic dichotomy is used earlier, only
for the conditional ergodic measures in the proof of rigidity.
\end{remark}
\subsection{A restricted-product measure and an ergodic action}
\label{sec:restricted-product}

We construct a semifinite measure whose translation modulus records the
sum of an absolutely summable sequence. Throughout this subsection, sequences
are indexed by $\N=\{0,1,\ldots\}$.
Set
\[
 I=[-\log 2,0],\qquad dm(t)=e^{-t}\,dt,\qquad \rho=m\mathbin{\upharpoonright}I.
\]
Thus $\rho$ is a probability measure, with density between $1$ and $2$ on
$I$. Put $S=I^{\N}$ and, for $N\geq0$, let
\[
 X_N=\R^{\{0,\ldots,N-1\}}\times I^{\{N,N+1,\ldots\}},
 \qquad X=\bigcup_{N\geq0}X_N.
\]
Regard
\[
 \sigma_N=m^{\otimes N}\otimes\rho^{\otimes\{N,N+1,\ldots\}}
\]
as a measure on $\R^{\N}$ carried by $X_N$. Since
$\sigma_{N+1}\mathbin{\upharpoonright}X_N=\sigma_N$, these measures increase.
Their limit
\[
 \sigma(B)=\lim_{N\to\infty}\sigma_N(B)
\]
is a measure and satisfies $\sigma\mathbin{\upharpoonright}X_N=\sigma_N$.
Indeed, countable additivity follows by interchanging two increasing sums
of nonnegative numbers. The sets
\[
 [-r,r]^N\times I^{\{N,N+1,\ldots\}},\qquad N,r\geq1,
\]
have finite measure and cover $X$. Consequently $\sigma$ is
$\sigma$-finite. It is diffuse and infinite, and
\[
 \sigma(S)=1,\qquad \sigma\mathbin{\upharpoonright}S=\rho^{\otimes\N}.
\]
For $h\in\R^{\N}$, write $T_h(x)=x+h$ and
$\sigma_h=(T_h)_*\sigma$.

\begin{lemma}
\label{lem:restricted-translation}
For every $h\in\R^{\N}$,
\[
 \begin{cases}
 \displaystyle\sigma_h=\exp\!\left(\sum_{n=0}^{\infty}h_n\right)\sigma,
       & h\in\ell_1(\R),\\[4pt]
 \sigma_h\perp\sigma,&h\notin\ell_1(\R).
 \end{cases}
\]
\end{lemma}

\begin{proof}
If $c$ has finite support, then $m(B-c_n)=e^{c_n}m(B)$ in each translated
coordinate. Applying this identity to $\sigma_N$ for $N$ beyond the support
of $c$ and then taking the increasing limit gives
\begin{equation}
 \sigma_c=e^{\sum_n c_n}\sigma.
 \label{eq:finite-shift-modulus}
\end{equation}

For the support calculation, define
\[
 b(u)=\rho\{t\in I:t+u\notin I\}.
\]
If $|u|<\log2$, the departing interval has length $|u|$, and the density
of $\rho$ on that interval lies between $1$ and $2$. If $|u|\geq\log2$,
then $b(u)=1$, up to irrelevant endpoints. There are therefore constants
$c,C>0$ such that
\begin{equation}
 c\min\{1,|u|\}\leq b(u)\leq C\min\{1,|u|\}.
 \label{eq:departure-estimate}
\end{equation}
Under $\sigma_N$, the coordinates after $N$ are independent with law
$\rho$. If $h\in\ell_1$, the first Borel--Cantelli lemma shows that
$x_n+h_n$ leaves $I$ only finitely often for almost every tail. Fubini's
theorem, followed by the increasing union over $N$, gives
$\sigma_h(X^c)=0$. If $h\notin\ell_1$, the sum of the departure
probabilities is infinite by \eqref{eq:departure-estimate}. Independence
and the second Borel--Cantelli lemma give infinitely many departures
almost surely, for every $\sigma_N$. Hence $\sigma_h(X)=0$. As $\sigma$
is carried by $X$, this proves the singularity assertion.

Suppose now that $h\in\ell_1$. For $N\geq0$, the product formula gives
\[
 \sigma_N(S-h)
   =\exp\!\left(\sum_{n<N}h_n\right)
        \prod_{n\geq N}\rho(I-h_n).
\]
The infinite product here denotes the probability of the intersection
of the independent tail conditions. Since
$\sum_n(1-\rho(I-h_n))<\infty$, its tail tends to $1$; this also follows
directly from the union bound
\[
 1-\prod_{n\geq N}\rho(I-h_n)
       \leq\sum_{n\geq N}(1-\rho(I-h_n)).
\]
It follows that
\begin{equation}
 \sigma_h(S)=\exp\!\left(\sum_n h_n\right).
 \label{eq:shifted-base-mass}
\end{equation}
More generally, let $B_i\subseteq I$ be Borel for $i<M$, and set
\[
 B=\prod_{i<M}B_i\times I^{\{M,M+1,\ldots\}}\subseteq S.
\]
For $N\geq M$, the same computation yields
\[
 \sigma_N(B-h)=
 \exp\!\left(\sum_{n<N}h_n\right)
 \left(\prod_{i<M}m(B_i)\right)
 \prod_{n\geq N}\rho(I-h_n).
\]
Taking limits proves the desired identity on this cylinder
$\pi$-system. Both restrictions to $S$ are finite, by
\eqref{eq:shifted-base-mass}, so uniqueness of finite measures gives
\[
 \sigma_h\mathbin{\upharpoonright}S
   =\exp\!\left(\sum_n h_n\right)
       \sigma\mathbin{\upharpoonright}S.
\]

Translations commute, so \eqref{eq:finite-shift-modulus} also implies
$(\sigma_h)_c=e^{\sum c_n}\sigma_h$ for every finitely supported $c$.
The preceding equality therefore holds on every translate $S+c$ with
$c\in\Q^{(\N)}$. These countably many translates cover $X$: each of
the finitely many coordinates outside $I$ can be moved into $I$ by a
rational translation. Disjointifying this cover proves the equality
on $X$. Both measures are carried by $X$, completing the proof.
\end{proof}

Let
\[
 D_{\R}=\left\{d\in\Q^{(\N)}:\sum_n d_n=0\right\}.
\]
This is the real-coordinate version of the subgroup $D=j(D_{\R})$
introduced earlier.

\begin{lemma}
\label{lem:restricted-ergodicity}
The translation action of $D_{\R}$ on $(X,\sigma)$ is free, measure
preserving, and ergodic.
\end{lemma}

\begin{proof}
Freeness is immediate, since a nonzero translation has no fixed point.
Measure preservation follows from \eqref{eq:finite-shift-modulus}.
Let $E\subseteq X$ be a Borel set invariant modulo $\sigma$ under
$D_{\R}$, and write $f=1_E$. We first show that $f$ is constant almost
everywhere on $S$.

Fix $N\geq2$. Restricting to $X_N$ and applying Fubini's theorem, we find
that for almost every tail under the product law, the resulting measurable function of
the first $N$ real variables is invariant, modulo Lebesgue null sets,
under every rational vector of coordinate sum zero. There is one
exceptional null set for this assertion because the rational group is
countable. The measure $m^{\otimes N}$ and Lebesgue measure have the same
null sets.

Such a function depends almost everywhere only on the sum of its
coordinates. To justify this explicitly, use linear coordinates
consisting of the sum and $N-1$ of the original coordinates. On almost
every sum fiber, Fubini's theorem gives invariance under the dense
rational subgroup of $\R^{N-1}$. A bounded measurable function with
this invariance is almost everywhere constant: convolution with a
compactly supported smooth approximate identity gives continuous
functions invariant under the rational subgroup, hence under all
translations, and these convolutions converge locally in $L^1$ to the
original function. This proves the assertion on each sum fiber.

In particular, $f\mathbin{\upharpoonright}S$ is invariant modulo
$\rho^{\otimes\N}$ under every finite permutation of the coordinates.
Here is the elementary zero--one argument needed for this conclusion.
If an event $F$ in an iid product is invariant under finite permutations,
choose a finite-head event $B$ with
$\mathbb P(F\mathbin{\triangle}B)<\varepsilon$. A finite permutation
moving that head to a disjoint block gives an event $B'$ independent of
$B$, with the same probability and with
$\mathbb P(F\mathbin{\triangle}B')<\varepsilon$. Thus
\[
 \bigl|\mathbb P(F)-\mathbb P(B)^2\bigr|<2\varepsilon.
\]
As $\varepsilon$ tends to zero, this gives
$\mathbb P(F)=\mathbb P(F)^2$. The argument only uses invariance modulo
null sets. Applying it to $E\cap S$ shows that $f$ equals one constant
in $\{0,1\}$ almost everywhere on $S$.

It remains to show that $D_{\R}S$ is conull in $X$. Choose a closed
interval $J\subset\operatorname{int}(I)$ of positive length. For
$\sigma$-almost every $x\in X$, infinitely many tail coordinates lie
in $J$. Indeed, this holds in the independent tail of each $\sigma_N$,
and Fubini's theorem and the increasing union prove it for $\sigma$.
For such an $x$, choose $N$ so that $x_n\in I$ for $n\geq N$.
Choose rational increments in the first $N$ coordinates which move them
into $\operatorname{int}(I)$, and let $s\in\Q$ be their sum. Choose
$M$ further coordinates lying in $J$, with
\[
 |s|/M<\operatorname{dist}(J,\partial I),
\]
and add $-s/M$ to each. The resulting translation has finite rational
support, has sum zero, and moves $x$ into $S$. This proves the conullity
of $D_{\R}S$. Since the group is countable and measure preserving, the
constant value of $f$ on $S$ propagates to a conull subset of $X$.
Hence the action is ergodic.
\end{proof}

We now push $\sigma$ forward by the continuous injective homomorphism
$j:\R^{\N}\to K$ and continue to denote the resulting measure by
$\sigma$. This is a $\sigma$-finite Borel measure on $K$: the image of
every set in the explicit countable finite-measure cover above is Borel,
by Lusin--Souslin. Its support in the measure-theoretic sense is contained
in the Borel set $j(X)$. The set $j(\R^{\N})$ is a Borel subgroup of $K$.
If $k\notin j(\R^{\N})$, translating $j(X)$ by $k$ moves it into a
disjoint coset of this subgroup. Combining this observation with
Lemma~\ref{lem:restricted-translation}, we obtain
\begin{equation}
 \QI(\sigma)=G=j(\ell_1(\R)),\qquad
 \sigma_{j(h)}=e^{\sum_n h_n}\sigma\quad(h\in\ell_1(\R)),
 \label{eq:compact-measure-modulus}
\end{equation}
and $\sigma_k\perp\sigma$ for every $k\notin G$. By
Lemma~\ref{lem:restricted-ergodicity}, the action of $D$ on $(K,\sigma)$
is free, measure preserving, and ergodic. These assertions concern the
measure space carried by $j(X)$; no local finiteness of $\sigma$ in the
compact topology of $K$ is asserted.

\subsection{Finite-corner representation coding}
\label{sec:finite-corner}

Let $D$ act on $C(K)$ by $\theta_d(f)(z)=f(z-d)$, and form the full
crossed product
\[
 A=C(K)\rtimes_\theta D.
\]
It is a separable unital $C^*$-algebra, since $K$ is compact metrizable
and $D$ is countable. The canonical covariant representation on the
measure space above gives a homomorphism
\[
 \iota_A:A\longrightarrow M,
 \qquad M=L^\infty(K,\sigma)\rtimes D.
\]
We do not need $\iota_A$ to be injective. Its image is ultraweakly dense
in $M$: the continuous functions generate the Borel sigma-algebra of
$K$, so their multiplication operators generate $L^\infty(K,\sigma)$,
and the image also contains the canonical group unitaries $u_d$.

The standard group--measure-space construction supplies a faithful
normal conditional expectation $E:M\to L^\infty(K,\sigma)$.
The freeness and ergodicity established above imply that $M$ is a
factor. Moreover,
\[
 \tau(b)=\int_K E(b)\,d\sigma,\qquad b\in M_+,
\]
is a faithful normal semifinite trace. Since the underlying measure is
diffuse, invariant, and infinite, $M$ is of type $\mathrm{II}_\infty$.
These are precisely the $\sigma$-finite versions of the crossed-product
facts in \cite[Theorems~5.2.2 and~5.2.4]{Peterson}. In particular, no
probability-preserving hypothesis is being substituted for the infinite
measure-preserving action used here.

Let
\[
 p=1_{j(S)}\in L^\infty(K,\sigma)\subseteq M.
\]
Then $\tau(p)=1$. On $L^2(M,\tau)$ write $L_b$ and $R_b$ for left
and right multiplication by $b\in M$. The projection onto the right
corner is $R_p$, and we put
\[
 \mathcal H=R_pL^2(M,\tau)=L^2(M,\tau)p.
\]
This is a fixed separable Hilbert space. Define
\[
 \lambda_p(b)=L_b\mathbin{\upharpoonright}\mathcal H\quad(b\in M),
 \qquad \pi=\lambda_p\circ\iota_A.
\]
The normal representation $\lambda_p$ is faithful: its kernel is an
ultraweakly closed two-sided ideal of the factor $M$, and it is nonzero
as a representation. The trace commutation theorem and compression to the right
corner give the following identity; see \cite[Remark~5.15]{Hiai20}
and \cite[Lemma~3.3.8]{Peterson}:
\[
 \lambda_p(M)'=R_pR(M)R_p\cong(pMp)^{\mathrm{op}}.
\]
Since $pMp$ is a $\mathrm{II}_1$ factor, $\mathcal H$ is infinite
dimensional, $\pi(A)''\cong M$ is of type $\mathrm{II}_\infty$, and
$\pi(A)'$ is of type $\mathrm{II}_1$.

For each $k\in K$, define an automorphism of $A$ by
\[
 \alpha_k(f)(z)=f(z+k),\qquad \alpha_k(u_d)=u_d.
\]
Translations commute with the $D$-action, so these formulas respect
the covariance relations and define an automorphism of the full crossed
product. They form a point-norm continuous action: continuity holds
uniformly on $K$ for each $f\in C(K)$, then for finite sums $\sum_d f_du_d$,
and finally for all of $A$ by norm approximation. Set
\begin{equation}
 \pi_k=\pi\circ\alpha_k\in\Rep(A,\mathcal H).
 \label{eq:fixed-module-family}
\end{equation}
Thus the family is point-norm continuous on one fixed Hilbert space.
Choosing one orthonormal basis identifies $\mathcal H$ with the usual
separable infinite-dimensional Hilbert space. There is no choice of a
Hilbert space, trace, or projection depending on $k$ in
\eqref{eq:fixed-module-family}.

\begin{lemma}
\label{lem:diagonal-spectral-class}
The maximal spectral measure class of
$\pi_k\mathbin{\upharpoonright}C(K)$ is $[\sigma_k]$.
In particular, if $k\notin G$, then $\pi_k$ and $\pi$ are disjoint.
\end{lemma}

\begin{proof}
For $k=0$, the spectral projection of a Borel set $B\subseteq K$ is
$\lambda_p(1_B)$. Faithfulness of $\lambda_p$ implies that it is zero
exactly when $\sigma(B)=0$. Explicitly, if $(\xi_r)$ is an orthonormal
basis of $\mathcal H$ and $P$ is this spectral measure, the probability
\[
 \nu(B)=\sum_{r=0}^{\infty}2^{-r-1}\langle P(B)\xi_r,\xi_r\rangle
\]
has exactly the null sets of $P$, hence exactly the null sets of $\sigma$.
This identifies the maximal spectral class even though $\sigma$ itself
is infinite. Under $\alpha_k$, the spectral projection of $B$ becomes
$P(B-k)$, whose null-set criterion is $\sigma_k(B)=0$.

If $k\notin G$, equation~\eqref{eq:compact-measure-modulus} gives
$\sigma_k\perp\sigma$. Any bounded intertwiner of the two
$A$-representations intertwines their restrictions to $C(K)$ and hence
their Borel spectral projections. A Borel set supporting one of the
two singular spectral classes and null for the other then forces this
intertwiner to be zero. Thus the representations are disjoint.
\end{proof}

\begin{lemma}
\label{lem:normal-trace-modulus}
For $h\in\ell_1(\R)$, the automorphism $\alpha_{j(h)}$ extends through
$\iota_A$ to a normal automorphism $\beta_h$ of $M$, and
\begin{equation}
 \tau\circ\beta_h=c_h\tau,
 \qquad c_h=\exp\!\left(\sum_n h_n\right).
 \label{eq:trace-scaling}
\end{equation}
The unitary implementing $\beta_h$ on the trace Hilbert space is
\begin{equation}
 U_h(a)=c_h^{-1/2}\beta_h(a),
 \qquad a\in M\cap L^2(M,\tau).
 \label{eq:explicit-standard-unitary}
\end{equation}
It conjugates both $L_b$ and $R_b$ to multiplication by $\beta_h(b)$.
\end{lemma}

\begin{proof}
By \eqref{eq:compact-measure-modulus}, translation by $j(h)$ preserves
the measure class of $\sigma$. It therefore defines a normal
automorphism of $L^\infty(K,\sigma)$, and it commutes with $D$.
It extends normally to the crossed product, with
\[
 \beta_h(f)(z)=f(z+j(h)),\qquad \beta_h(u_d)=u_d.
\]
One can check normality directly in the regular covariant realization
on $L^2(K,\sigma)\otimes\ell_2(D)$: simultaneous translation of the
$K$ variable, multiplied by $c_h^{-1/2}$, is a unitary and conjugates
the diagonal and group generators according to the displayed formulas.

The conditional expectation commutes with these automorphisms, first
on finite Fourier sums and then on $M$ by normality. Consequently,
for $b\in M_+$,
\[
 \tau(\beta_h(b))
   =\int_K E(b)(z+j(h))\,d\sigma(z)
   =c_h\int_K E(b)(z)\,d\sigma(z),
\]
which proves \eqref{eq:trace-scaling}. In particular, the sign of the
exponent agrees with the convention $\sigma_{j(h)}=(T_{j(h)})_*\sigma$.

For $a\in M\cap L^2(M,\tau)$, equation~\eqref{eq:trace-scaling} gives
\[
 \|c_h^{-1/2}\beta_h(a)\|_2^2
   =c_h^{-1}\tau(\beta_h(a^*a))=\tau(a^*a).
\]
The inverse formula using $\beta_h^{-1}$ proves that this isometry
extends to a unitary on $L^2(M,\tau)$. Multiplication on the dense
square-integrable ideal gives
\[
 U_hL_bU_h^*=L_{\beta_h(b)},\qquad
 U_hR_bU_h^*=R_{\beta_h(b)}.
\]
Thus \eqref{eq:explicit-standard-unitary} has all the asserted properties.
\end{proof}

\begin{lemma}
\label{lem:right-corner-comparison}
For nonzero finite projections $p_1,p_2\in M$, the left representations
of $M$ on $L^2(M,\tau)p_1$ and $L^2(M,\tau)p_2$ are unitarily equivalent
if and only if $\tau(p_1)=\tau(p_2)$. The same criterion holds after
restriction along $\iota_A$.
\end{lemma}

\begin{proof}
A unitary intertwiner between the two left modules, extended by zero
on the orthogonal complement of its initial module, is a partial
isometry on $L^2(M,\tau)$ commuting with $L(M)$. By the left--right
commutation theorem, it belongs to $R(M)$. Its initial and final
projections are $R_{p_1}$ and $R_{p_2}$, so it gives Murray--von Neumann
equivalence of $p_1$ and $p_2$ in $M$. Conversely, a partial isometry
$v\in M$ with $v^*v=p_1$ and $vv^*=p_2$ gives the required unitary by
right multiplication by $v^*$. Finite projections in a semifinite
factor are Murray--von Neumann equivalent precisely when their traces
are equal; see \cite[Theorem~2.7 and Corollary~2.8(1)]{Sherman}.
Finally, an intertwiner for the restrictions to $A$ intertwines $M$ as
well, because the representations are normal and $\iota_A(A)$ is
ultraweakly dense. This proves the last assertion.
\end{proof}

\begin{theorem}
\label{thm:representation}
For all $k,l\in K$,
\[
 \pi_k\,\UE\,\pi_l\quad\Longleftrightarrow\quad k-l\in H.
\]
Moreover, the quasi-equivalence stabilizer of $\pi$ is $G$, and its
unitary-equivalence stabilizer is $H$. Every representation in the
continuous family \eqref{eq:fixed-module-family} has represented factor
of type $\mathrm{II}_\infty$ and commutant of type $\mathrm{II}_1$.
\end{theorem}

\begin{proof}
Lemma~\ref{lem:diagonal-spectral-class} excludes quasi-equivalence,
and in particular unitary equivalence, when $k\notin G$. If $k=j(h)$
with $h\in\ell_1(\R)$, Lemma~\ref{lem:normal-trace-modulus} gives a
normal automorphism $\beta_h$ of $M$ satisfying
$\beta_h\circ\iota_A=\iota_A\circ\alpha_k$. It follows that
$\pi_k$ is quasi-equivalent to $\pi$.

To determine unitary equivalence, conjugate by the unitary $U_h^*$
from \eqref{eq:explicit-standard-unitary}. It maps
$L^2(M,\tau)p$ onto $L^2(M,\tau)\beta_h^{-1}(p)$ and identifies
$\pi_k$ with the ordinary left representation on that latter module.
The two right projections have traces
\[
 \tau(p)=1,\qquad
 \tau(\beta_h^{-1}(p))=c_h^{-1}
       =\exp\!\left(-\sum_n h_n\right).
\]
By Lemma~\ref{lem:right-corner-comparison}, the modules are equivalent
exactly when $\sum_n h_n=0$, that is, exactly when $k\in H$.
This proves both stabilizer assertions.

Precomposing both members of a pair with $\alpha_{-l}$ preserves
unitary equivalence and reduces the pair $(\pi_k,\pi_l)$ to
$(\pi_{k-l},\pi)$. This proves the displayed equivalence for all pairs.
Finally, $\alpha_k$ is onto, so $\pi_k(A)=\pi(A)$ for every $k$.
Their von Neumann closures and commutants are therefore the fixed
algebras already identified above.
\end{proof}
\subsection{Coding by a single bounded operator}\label{sec:coding}
We give the explicit coding needed to pass from the representation family
to the two questions in the introduction. Throughout this subsection all
identifications of fixed separable infinite-dimensional Hilbert spaces
are chosen once and do not depend on the input.

\begin{lemma}\label{lem:tuple-coding}
Simultaneous unitary equivalence of countable tuples of unitaries
continuously reduces to unitary equivalence of operators of norm at most
$7/8$ on a fixed Hilbert space. Here the source has the product strong
topology and the target ball has the strong operator topology (or the
strong-star topology).
\end{lemma}
\begin{proof}
Set $\mathcal H'=\bigoplus_{j\ge0}\mathcal H_j$, where each summand is a
copy of $\mathcal H$, and let $J_j:\mathcal H\to\mathcal H'$ be its
inclusion. For a tuple $U=(U_n)_{n\ge0}$ define
\[
 \Delta=\sum_{j\ge0}2^{-j-1}J_jJ_j^*
\]
and
\begin{align*}
 B_U={}&\sum_{j\ge1}2^{-j-4}(J_0J_j^*+J_jJ_0^*)\\
 &+\sum_{n\ge0}2^{-n-4}
 (J_{2n+1}U_nJ_{2n+2}^*+J_{2n+2}U_n^*J_{2n+1}^*),\\
 T_U={}&\Delta+iB_U.
\end{align*}
The series converge in norm, uniformly in $U$, and
$\|B_U\|\le1/8+1/4=3/8$. Thus $\|T_U\|\le7/8$.
Since $B_U$ is self-adjoint, the real and imaginary parts of $T_U$
are $\Delta$ and $B_U$.

A unitary $W$ intertwining every $U_n$ and $V_n$ gives an intertwiner
$\bigoplus_j W$ between $T_U$ and $T_V$.
Conversely, a unitary $Q$ with $QT_UQ^*=T_V$ must commute with $\Delta$.
Its distinct eigenvalues $2^{-j-1}$ have eigenspaces exactly
$\mathcal H_j$, whose direct sum is all of $\mathcal H'$. Consequently
$Q=\bigoplus_j W_j$. The $(0,j)$ block of $QB_UQ^*=B_V$ gives
$W_0W_j^*=I$, so $W_j=W_0$ for all $j$. The $(2n+1,2n+2)$ block then
gives $W_0U_nW_0^*=V_n$ for every $n$.

On the unitary group the adjoint is strongly continuous. Each finite
partial sum is therefore continuous in the product strong topology,
and the omitted tails have a uniform norm bound tending to zero.
This proves the claimed continuity. Finally identify $\mathcal H'$
with $\mathcal H$ by one fixed unitary.
\end{proof}

\begin{lemma}\label{lem:algebra-coding}
For every separable unital $C^*$-algebra $A$, unitary equivalence of
representations of $A$ on $\mathcal H$ Borel reduces to
$\approx^+_{F_\infty}$ and to ${\simeq_u}$.
\end{lemma}
\begin{proof}
Choose a countable generating sequence of unitaries $(a_n)$ in $A$.
Such a sequence exists: a separable unital $C^*$-algebra has a countable
generating sequence of self-adjoint contractions $b_n$, and the
unitaries $e^{ib_n}$ generate the same algebra, since the exponential
has a continuous inverse on the arc $\{e^{it}:-1\le t\le1\}$.
A representation $\pi$ is determined by $(\pi(a_n))_n$, and a unitary
intertwines two representations exactly when it intertwines these
tuples. A countable tuple of unitaries is precisely a representation
of the free group $F_\infty$ on its named generators.
Evaluation at fixed elements is Borel, indeed continuous in the
pointwise strong topology on representations. Lemma~\ref{lem:tuple-coding}
gives the second reduction. For the family in
Theorem~\ref{thm:representation}, these maps are continuous.
\end{proof}

\begin{proposition}\label{prop:operator-group}
The relations ${\simeq_u}$ and $\approx^+_{F_\infty}$ are
Borel bireducible. They are also Borel bireducible with simultaneous
unitary equivalence of pairs of unitaries.
\end{proposition}
\begin{proof}
Only the reduction from one arbitrary bounded operator to two
unitaries remains to be given. Put
\[
 Q_T=\frac{T}{1+\|T\|},\qquad
 A_T=\frac{Q_T+Q_T^*}{2},\qquad
 C_T=\frac{Q_T-Q_T^*}{2i},
\]
and send $T$ to $(e^{iA_T},e^{iC_T})$.
The two self-adjoint operators have norm less than one. Functional
calculus with the fixed continuous inverse of the exponential on the
arc above recovers $A_T$ and $C_T$ from these unitaries. Thus simultaneous
conjugacy recovers $Q_T$, and then recovers
\[
 T=\frac{Q_T}{1-\|Q_T\|}.
\]
Conversely every conjugacy of $T$ gives a conjugacy of the pair.

The map is Borel: the norm and adjoint are Borel in the coefficient
Borel structure, and the exponential on a bounded self-adjoint ball
is a strong limit of polynomials with uniform error bounds.
Pad the pair by identities to obtain a countable unitary tuple.
The reverse reduction is Lemma~\ref{lem:tuple-coding}.
\end{proof}

\subsection{The operator and representation theorems}
\begin{proof}[Proof of Theorem~\ref{main:operators}]
Theorem~\ref{thm:representation} and Lemma~\ref{lem:algebra-coding} give
\[
 K/H\leB{\simeq_u},
\]
whereas Theorem~\ref{thm:nonreduction} gives $K/H\nleB{\equiv_m}$.
Composition of reductions therefore proves
$ {\simeq_u}\nleB{\equiv_m}$.

For the reverse reduction, embed Cantor space in the unit circle and
send $\mu\in\Prob(2^\N)$ to the infinite amplification of multiplication
by the circle coordinate on $L^2(\mu)$. Its multiplicity is uniformly
infinite and its maximal spectral measure class is $[\mu]$, so two such
unitaries are equivalent exactly when the measures are equivalent.
This has a Borel realization on one fixed Hilbert space: apply
Gram--Schmidt, skipping zero vectors, to a fixed countable dense family
of continuous functions tensored with the standard basis of $\ell_2$.
All inner products are Borel functions of $\mu$, and the amplified
Hilbert space is always infinite dimensional. The resulting multiplication
operators therefore have Borel matrix coefficients.
This proves $ {\equiv_m}\leB{\simeq_u}$, and hence strictness.
\end{proof}

\begin{proof}[Proof of Theorem~\ref{main:thomas}]
A representation of $\mathbb Z$ is determined by one unitary.
The spectral theorem, with spectral multiplicities encoded by tagged
measures, gives
\[
 {\approx^+_{\mathbb Z}}\sim_B{\equiv_m};
\]
see Gao \cite[Section~8.5, p.~200]{Gao} and Thomas
\cite[Section~8]{Thomas}. Proposition~\ref{prop:operator-group} gives
\[
 {\approx^+_{F_\infty}}\sim_B{\simeq_u}.
\]
Theorem~\ref{main:operators} now supplies the required strict inequality.
In particular, arbitrary representations of $F_\infty$ cannot be
classified, by a Borel assignment, using measure-equivalence classes.
\end{proof}

\section*{Acknowledgment}
This paper is a collaboration with artificial intelligence. The AI model is gpt 6. The AI proves the main theorems under human suggestions and writes early drafts. The author revises the paper and improves readability. The author has reviewed all AI-generated content.

\begingroup
\raggedright

\endgroup
\noindent\begin{minipage}{\textwidth}
\small
\textsc{School of Mathematical Sciences and School of Pre-university,}\\
\textsc{Dalian Minzu University}\\[0.4em]
\textit{Email address:}\enspace
\href{mailto:wangruijun@dlnu.edu.cn}{\texttt{wangruijun@dlnu.edu.cn}}
\end{minipage}
\end{document}